\documentclass{amsart}

\usepackage[mathscr]{eucal}
\usepackage{amssymb}
\usepackage[usenames,dvipsnames]{xcolor}
\usepackage[normalem]{ulem}
\usepackage{amsthm}
\usepackage{bbold}
\usepackage{enumerate}
\usepackage{array}
\usepackage{amsmath}

\numberwithin{equation}{section}
\usepackage[all]{xy}
\newdir{ >}{{}*!/-10pt/\dir{>}}

\usepackage[colorlinks=true,linkcolor={Brown},citecolor={Brown},urlcolor={Brown}]{hyperref}

\usepackage{cleveref}

\swapnumbers 

\newtheorem{Cor}[equation]{Corollary}
\newtheorem{Lem}[equation]{Lemma}
\newtheorem{Prop}[equation]{Proposition}
\newtheorem{Thm}[equation]{Theorem}

\theoremstyle{remark}
\newtheorem{Def}[equation]{Definition}
\newtheorem{Defs}[equation]{Definitions}
\newtheorem{Not}[equation]{Notation}
\newtheorem{Exa}[equation]{Example}
\newtheorem{Exas}[equation]{Examples}
\newtheorem{Cons}[equation]{Construction}
\newtheorem{Conv}[equation]{Convention}
\newtheorem{Hyp}[equation]{Hypothesis}
\newtheorem{War}[equation]{Warning}
\newtheorem{Rem}[equation]{Remark}
\newtheorem{Rec}[equation]{Recollection}

\newcommand{\nc}{\newcommand}
\nc{\dmo}{\DeclareMathOperator}

\dmo{\Ab}{Ab}
\dmo{\add}{add}
\dmo{\Ann}{Ann}
\dmo{\Aut}{Aut}
\dmo{\BIK}{\mathsf{Rel}}
\dmo{\bik}{\mathsf{rel}}
\dmo{\cbrk}{\mathrm{CBrk}}
\dmo{\CIdim}{CI-dim}
\dmo{\codim}{codim}
\dmo{\coev}{coev}
\dmo{\Coh}{Coh}
\dmo{\coh}{coh}
\dmo{\coker}{coker}
\dmo{\cone}{cone}
\dmo{\Cov}{Cov}
\dmo{\cx}{cx}
\dmo{\depth}{depth}
\dmo{\Der}{D}
\dmo{\dgModu}{dg-Mod}
\dmo{\dlim}{\underrightarrow{\mathrm{lim}}}
\dmo{\End}{End}
\dmo{\ev}{ev}
\dmo{\Ext}{Ext}
\dmo{\fdim}{fd}
\dmo{\Flat}{Flat}
\dmo{\fp}{fp}
\dmo{\GFlat}{GFlat}
\dmo{\Gh}{Gh}
\dmo{\GInj}{GInj}
\dmo{\gldim}{gl. dim}
\dmo{\GProj}{GProj}
\dmo{\Gr}{\mathsf{Gr}}
\dmo{\gr}{\mathsf{gr}}
\dmo{\height}{ht}
\dmo{\HN}{HN}
\dmo{\hocolim}{hocolim}
\dmo{\Hom}{Hom}
\dmo{\Ident}{Id}
\dmo{\idim}{id}
\dmo{\Id}{Id}
\dmo{\id}{id}
\dmo{\ilim}{\underleftarrow{\mathrm{lim}}}
\dmo{\Img}{Im}
\dmo{\im}{im}
\dmo{\Ind}{Ind}
\dmo{\ind}{ind}
\dmo{\Inj}{Inj}
\dmo{\inj}{inj}
\dmo{\IPI}{Sp}
\dmo{\ipi}{sp}
\dmo{\kdim}{\mathrm{kdim}}
\dmo{\Ker}{Ker}
\dmo{\level}{\mathrm{level}}
\dmo{\Loc}{Loc}
\dmo{\loc}{\mathrm{loc}}
\dmo{\MCM}{MCM}
\dmo{\mcone}{cone}
\dmo{\modname}{mod}%
\dmo{\Modu}{\mathsf{Mod}}
\dmo{\modu}{\mathsf{mod}}
\dmo{\Mod}{Mod}
\dmo{\nil}{nil}
\dmo{\ob}{Ob}
\dmo{\opname}{op}
\dmo{\pdim}{pd}
\dmo{\Perf}{\mathrm{Perf}}
\dmo{\Ph}{Ph}
\dmo{\Pic}{Pic}
\dmo{\Proj}{Proj}
\dmo{\proj}{proj}
\dmo{\Psh}{Psh}
\dmo{\QCoh}{QCoh}
\dmo{\Qcoh}{Qcoh}
\dmo{\qGr}{\mathsf{qGr}}
\dmo{\qgr}{\mathsf{qgr}}
\dmo{\rank}{rank}
\dmo{\restrict}{res}
\dmo{\res}{res}
\dmo{\RHom}{\mathbf{R}Hom}
\dmo{\rk}{rk}
\dmo{\Rname}{R}
\dmo{\RsHom}{\mathbf{R}\mathcal{H}\mathit{om}}
\dmo{\Schs}{Sch/S}
\dmo{\Sch}{Sch}
\dmo{\sExt}{\mathcal{E}\mathit{xt}}
\dmo{\sGInj}{\underline{\GInj}}
\dmo{\sGr}{\underline{\mathsf{Gr}}}
\dmo{\sgr}{\underline{\mathsf{gr}}}
\dmo{\Shn}{Sh_{Nis}}
\dmo{\sHom}{\mathcal{H}\mathit{om}}
\dmo{\Sh}{Sh}
\dmo{\SH}{SH}
\dmo{\Sing}{Sing}
\dmo{\smallb}{b}
\dmo{\sMCM}{\underline{\MCM}}
\dmo{\smodname}{\underline{mod}}%
\dmo{\sModu}{\underline{\mathsf{Mod}}}
\dmo{\smodu}{\underline{\mathsf{mod}}}
\dmo{\sMod}{\underline{Mod}}%
\dmo{\smod}{\underline{mod}}%
\dmo{\Spc}{Spc}
\dmo{\Spech}{Spec^h}
\dmo{\Spec}{Spec}
\dmo{\Split}{Split}
\dmo{\stbik}{\mathsf{strel}}
\dmo{\stBIK}{\mathsf{{Rel}}}
\dmo{\Stmodu}{\mathsf{StMod}}
\dmo{\stmodu}{\mathsf{stmod}}
\dmo{\Stmod}{\mathsf{StRel}}
\dmo{\stmod}{\mathsf{strel}}
\dmo{\supph}{supph}
\dmo{\supp}{supp}
\dmo{\thick}{thick}
\dmo{\Thick}{\mathrm{Thick}}
\dmo{\Th}{Th}
\dmo{\tlevel}{\mathrm{level}^\otimes}
\dmo{\Tor}{Tor}
\dmo{\Tot}{Tot}
\dmo{\uHom}{\underline{Hom}}
\dmo{\Vis}{Vis}

\nc\EE{\mathbb{E}}
\nc\mcC{\mathcal{C}}
\nc\mcH{\mathcal{H}}
\nc\mcI{\mathcal{I}}
\nc\mcN{\mathcal{N}}
\nc\mcQ{\mathcal{Q}}
\nc\mcS{\mathcal{S}}
\nc\scE{\mathscr{E}}
\nc\scS{\mathscr{S}}
\nc\ZZ{\mathbb{Z}}
\nc{\adjto}{\rightleftarrows}
\nc{\adj}{\dashv}
\nc{\aka}{{a.\,k.\,a.}\ }
\nc{\bbZ}{\mathbb{Z}}
\nc{\bpi}{\bar{\pi}}
\nc{\calI}{\mathcal{I}}
\nc{\caT}[1]{\cat{#1}}
\nc{\cat}[1]{\mathscr{#1}}
\nc{\cK}{\cat{K}}
\nc{\colim}{\mathop{\mathrm{colim}}}
\nc{\Db}{\Der^{\smallb}}
\nc{\EC}{\scE_{\cat{C}}}
\nc{\Edist}{\scE_{\textrm{\rm dist}}}
\nc{\eg}{{\sl e.g.}}
\nc{\Endcat}[1]{\End_{\cat #1}}
\nc{\Esplit}{\scE_{\textrm{\rm split}}}
\nc{\gm}{\mathfrak{m}}
\nc{\Greg}[1]{{\color{CarnationPink}#1}}
\nc{\Homcat}[1]{\Hom_{\cat #1}}
\nc{\hook}{\hookrightarrow}
\nc{\ideal}[1]{\langle #1\rangle}
\nc{\ie}{{i.e.}\ }
\nc{\ihom}{{\mathsf{hom}}} 
\nc{\iinj}{\,\text{-}\inj}%
\nc{\into}{\mathop{\rightarrowtail}}
\nc{\inv}{^{-1}}
\nc{\isotoo}{\overset{\sim}{\,\too\,}}
\nc{\isoto}{\overset{\sim}{\,\to\,}}
\nc{\kk}{\Bbbk}
\nc{\loccit}{{\sl loc.\ cit.}}
\nc{\matrice}[1]{\begin{pmatrix} #1 \end{pmatrix}}
\nc{\mcA}{\cat{A}}
\nc{\mcB}{\cat{B}}
\nc{\Mid}{\,\big|\,}
\nc{\mmod}{\modname\kern-0.1em\text{-}}%
\nc{\MMod}{\text{-}\kern-0.1em\Mod}%
\nc{\NE}{\mcN_{\kern-.1em\scE}}
\nc{\NF}{\mcN_{\kern-.1em F}}
\nc{\onto}{\mathop{\twoheadrightarrow}}
\nc{\op}{^{\opname}}
\nc{\otoo}[1]{\overset{#1}{\,\too\,}}
\nc{\oto}[1]{\overset{#1}\to}
\nc{\ourfrac}[2]{\genfrac{}{}{0pt}{}{\scriptstyle #1}{\scriptstyle #2}}
\nc{\oursetminus}{\!\smallsetminus\!}
\nc{\Paul}[1]{{\color{OliveGreen}#1}}
\nc{\potimes}[1]{^{\otimes #1}}
\nc{\Pout}[1]{\Paul{\sout{#1}}}
\nc{\pproj}{\,\text{-}\proj}%
\nc{\qquadtext}[1]{\qquad\textrm{#1}\qquad}
\nc{\quadtext}[1]{\quad\textrm{#1}\quad}
\nc{\Rcatb}[1]{\Rcat{#1}^\sbull}
\nc{\Rcat}[1]{\Rname_{\cat #1}}
\nc{\restr}[1]{_{|_{\scriptstyle #1}}}
\nc{\rfc}{(\mathbf{rf1^c})}
\nc{\rff}{(\mathbf{rf2})}
\nc{\rf}{(\mathbf{rf1})}
\nc{\RKb}{\Rcatb{K}}
\nc{\RKub}{\Rcat{K,u}^\sbull}
\nc{\RLvb}{\Rcat{L,v}^\sbull}
\nc{\sbull}{{\scriptscriptstyle\bullet}}
\nc{\SET}[2]{\big\{\,#1\Mid#2\,\big\}}
\nc{\SE}{\scS_{\scE}}
\nc{\SKB}{\STAB{\cat{K}}{B}}
\nc{\SKEp}{\STAB{\cat{K}}{\scE'}'}
\nc{\SKE}{\STAB{\cat{K}}{\scE}}
\nc{\SKF}{\STAB{\cat{K}}{F}}
\nc{\smat}[1]{\left(\begin{smallmatrix} #1 \end{smallmatrix}\right)}
\nc{\SpcK}{\Spc(\cat{K})}
\nc{\STAB}[2]{\underline{#1}_{#2}}
\nc{\then}{\Rightarrow}
\nc{\threestars}{\medbreak\begin{center}*\ *\ *\end{center}\medbreak}
\nc{\too}{\mathop{\longrightarrow}\limits}
\nc{\ttclass}{$\otimes$-class}
\nc{\tthick}{\thick^{\otimes}}
\nc{\ucone}{\underline{\cone}}
\nc{\unit}{\mathbb{1}}
\nc{\Verd}{/\!\!/}
\nc{\xto}[1]{\xrightarrow{#1}}

\begin{document}


\title{Relative stable categories and birationality}
\author{Paul Balmer}
\author{Greg Stevenson}
\date{2021 May 7}

\address{Paul Balmer, Mathematics Department, UCLA, Los Angeles, CA 90095-1555, USA}
\email{balmer@math.ucla.edu}
\urladdr{http://www.math.ucla.edu/~balmer}

\address{Greg Stevenson, School of Mathematics and Statistics,
University of Glasgow,
University Place,
Glasgow G12 8QQ
}
\email{gregory.stevenson@glasgow.ac.uk}
\urladdr{http://www.maths.gla.ac.uk/~gstevenson/}

\begin{abstract}
We propose a general method to construct new triangulated categories, \emph{relative stable categories}, as additive quotients of a given one. This construction enhances results of Beligiannis, particularly in the tensor-triangular setting. We prove a birationality result showing that the original category and its relative stable quotient are equivalent on some open piece of their spectrum.
\end{abstract}


\thanks{PB supported by NSF grant~DMS-1600032.}

\maketitle


\vskip-\baselineskip\vskip-\baselineskip
\tableofcontents


\section{Introduction}


Traditionally, given a triangulated category $\cK$, it is difficult to generate new triangulated categories from it. One can, of course, take Verdier quotients or pass to thick subcategories. However, this is quite restrictive and many properties of $\cK$ will be inherited by any thick subcategory or quotient, which can be a curse as much as a blessing.

In recent years there have been breakthroughs which allow us to produce new triangulated categories from a given one\textemdash{}without recourse to an enhancement. In \cite{Balmer11} the first-named author showed how to produce triangulated categories of modules over separable monads. Even more recently, Neeman has shown that one can complete a triangulated category with respect to a metric (see \cite{neeman2020metrics} for a survey of Neeman's results). This host of new examples have already had profound implications for our understanding of geometry and representation theory.

This article is built upon a similarly profound construction, due to Beligiannis~\cite{Beligiannis00}, which also yields new triangulated categories. Despite seeming relatively unknown, it predates the two constructions mentioned above. Beligiannis' insight is that by treating a triangulated category equipped with a well behaved class of triangles as if it were a Frobenius exact category one can stabilize to produce a new triangulated category.

It is not hard to produce such well behaved classes of triangles: We show in Theorem~\ref{thm:Wirthmuller} that any sufficiently nice triple of adjoints gives rise to one. There are many examples of such triples, and hence of such stabilizations. For instance any spherical functor gives rise to a stable category measuring how far away the spherical functor is from being faithful (\Cref{rem:Wirth}).

\threestars

For the rest of the introduction, and for most of the article, we focus on the tensor triangulated case. Let $B$ be a rigid (\aka strongly dualizable) object in a tensor-triangulated category~$\cK$. The \emph{stable category of~$\cK$ relatively to~$B$} is the additive quotient
\[
\SKB=\frac{\cK}{B\otimes\cK}
\]
that is, the category obtained by keeping the same objects as~$\cK$ and by modding out the morphisms that factor via a tensor-multiple of~$B$, \ie via an object of~$B\otimes\cK$. The name comes from the fact that objects~$X,Y$ of~$\cK$ become isomorphic in~$\SKB$ if and only if they become `stably' isomorphic $X\oplus P\simeq Y\oplus Q$ in~$\cK$ after adding direct summands $P,Q$ of objects in~$B\otimes \cK$. (See \Cref{lem:iso}.)

This quotient $\SKB$ should not be confused with the more standard Verdier quotient (localization) of~$\cK$ by the thick subcategory generated by~$B$: We quantify in \Cref{lem:notexact1} the extent to which these constructions disagree. Nonetheless, the category~$\SKB$ remains a \emph{tensor-triangulated category}, with triangulated structure first constructed by Beligiannis~\cite{Beligiannis00}, as discussed above.

Before analyzing~$\SKB$, let us emphasize the warm generality in which our discussion is basking. Among a much longer list of examples, here are three classics:
\begin{itemize}
\item[-]
$B$ could be a perfect complex in the derived category
of a scheme;
\item[-]
$B$ could be a finite genuine $\Gamma$-spectrum in the equivariant stable homotopy category of a compact Lie group~$\Gamma$ (even $\Gamma=1$ is interesting);
\item[-]
$B$ could be a finite-dimensional representation in the stable module category of a finite group.
\end{itemize}

Historically, it seems that relative stable categories first appeared in the third setting, \ie in modular representation theory, for instance in~\cite{CarlsonPengWheeler98}. In recent years, those categories have generated renewed interest, in connection with cluster algebras~\cite{Jorgensen10}, around Brou\'e's conjecture~\cite{WangZhang18} or even closer to tensor-triangular geometry in~\cite{BalandChirvasituStevenson19,BalandStevenson19,Carlson18}.

\medbreak

So what do we know about the tt-category $\SKB$ in general? Well, very little. In the recent~\cite[Section~5]{Carlson18}, Carlson shows that~$\SKB$ behaves in a `non-noetherian' fashion already in modular representation theory; for instance, $\SKB$ can contain infinite towers of thick tensor-ideals. In other words, the methods that were successful for the ordinary stable category~\cite{BensonCarlsonRickard97}, where cohomology is noetherian, have to give way to abstract tensor-triangular geometry: Thick tensor-ideals are classified by the spectrum~$\Spc(\SKB)$ of~\cite{Balmer05a} and one should try to understand that space.

It is natural to ask how this spectrum $\Spc(\SKB)$ relates to the spectrum~$\SpcK$ of the tt-category we start from. To understand the difficulty, note that the natural quotient functor $\pi\colon \cK\to \SKB$ is a tensor-functor but \emph{not a triangulated functor}. In fact, the suspensions in~$\cK$ and $\SKB$ differ by twisting with an invertible object~$F_B$ in~$\SKB$. So it does not even make sense to ask whether the image of a distinguished triangle in~$\cK$ remains distinguished in~$\SKB$, for it is not even a triangle! Consequently, we cannot produce a continuous map $\Spc(\SKB)\to \SpcK$ by using functoriality of~$\Spc(-)$ and we do not know whether such a map exists in general. (Compare \Cref{rem:pi^!} though.) Inspired by \cite{BalandStevenson19}, we can however prove a birationality result relating~$\cK$ and~$\SKB$.
\begin{Thm}[{Birationality Theorem~\ref{thm:birationality}}]
\label{thm:birationality-intro}%
There exists an object $C\in\SKB$ with the following property. Consider the open subsets $U=\SpcK\oursetminus\supp(B)$ and $V=\Spc(\SKB)\oursetminus\supp(C)$ and the respective localizations $\cK\to \cK(U)$ and $\SKB\to \SKB(V)$. Then there exists a commutative diagram
\[
\xymatrix{
\cK \ar[d]_-{\res_U} \ar[r]^-{\pi}
& \SKB \ar[d]^-{\res_V}
\\
\cK(U) \ar[r]^-{\cong}
& \SKB(V)
}
\]
in which the lower functor is a tensor-\emph{triangulated} equivalence $\cK(U)\cong \SKB(V)$.
In particular these open subspaces $U$ and $V$ are homeomorphic.
\end{Thm}

In other words, although the functor $\pi\colon \cK \to \SKB$ is not even triangulated, let alone an equivalence, there is a `common' open piece of the spectra $\SpcK$ and $\Spc(\SKB)$ on which $\pi$ becomes a triangulated equivalence.

\medbreak

Let us quickly outline the content of this work. After recalling basic terminology and notation in \Cref{se:prelims}, we spend a few pages discussing the general framework of stable categories, without assuming the existence of a tensor product. We hope our condensed presentation of the general theory will be useful to some readers beyond the above tt-geometric preoccupations. Specifically, \Cref{se:proper-E} recalls the notion of proper class of exact triangles and \Cref{se:Beligiannis} discusses the triangulation.

We also provide a pre-tensor version of the above birationality result in \Cref{se:thick}, see \Cref{cor:generalBS}, after an analysis of the thick triangulated subcategories of the stable category in terms of the corresponding (non-triangulated) subcategories of~$\cK$.

\Cref{se:algebraicity} is an interlude. We prove that if we start with an \emph{algebraic} triangulated category~$\cK$ then all stable categories built out of~$\cK$ remain algebraic.
Recall that a triangulated category $\cat{K}$ is algebraic, in the sense of Keller \cite{Keller2007differential}, if it can be presented as the stable category of a Frobenius exact category; an equivalent definition is that $\cat{K}$ can be presented as the homotopy category of a differential-graded category. Knowing that $\cat{K}$ is algebraic allows, for instance, the invocation of results from tilting theory.
\begin{Thm}[{\Cref{thm:algebraicity}}]
\label{thm:algebraicity-intro}
The relative stable category of an algebraic triangulated category remains algebraic.
\end{Thm}
This has been proved by Beligiannis in~\cite[Corollary~3.18]{Beligiannis13} under a strong Ext-vanishing condition. We give a proof in complete generality.

The tensor enters the paper only in \Cref{se:tt-stab}, where we discuss how the relative stable category inherits a tensor and we make other preparations. We are then ready for \Cref{se:birationality} where we wrap everything up and give the proof of \Cref{thm:birationality-intro}.

\bigbreak
\textbf{Acknowledgments}: We are thankful to Apostolos Beligiannis and Jon Carlson, for valuable discussions around relative categories and references.

\goodbreak
\section{Preliminaries}
\label{se:prelims}%
\medbreak

We begin by setting some conventions and introducing some basic notation which will be used throughout. We do not recall details of the ever expanding yoga of tt-geometry here; in fact, tensor products do not appear until Section~\ref{se:tt-stab} and we utilize only fairly minimal machinery. Nonetheless, for a gentle introduction the interested reader should consult~\cite{BalmerICM}.

Our conventions for the various notions of equivalence and isomorphism, and our standing closure conditions on subcategories, are relatively standard.

\begin{Conv}
We write $\cong$ for canonical (or natural) isomorphisms. We write $\simeq$ for a mere isomorphism for which we do not claim any naturality/uniqueness. We write = for strict equalities, unique isomorphisms, or when we are lying desperately.
\end{Conv}

\begin{Conv}
Throughout all subcategories are assumed to be full and replete, i.e.\ closed under isomorphisms, unless explicitly mentioned otherwise.
\end{Conv}

\begin{Conv}
Let $\cat{K}$ be a triangulated category and consider an exact triangle
\begin{displaymath}
X\otoo{f} Y \otoo{g} Z \otoo{h} \Sigma X
\end{displaymath}
in $\cat{K}$. We will call $Z=\cone(f)$ the cone of~$f$ (really the cone is the map $g$ but this is a standard abuse). We sometimes also say that $g$ (or just~$Z$) is the cofibre of~$f$, or that $f$ (or just~$X$) is the fibre of~$g$.
\end{Conv}

Now we come to some notation that will be important throughout. The first piece is again standard.

\begin{Not}
Given a set $\mcA$ of objects in $\cat{K}$ we denote by $\add(\mcA)$ the closure of $\mcA$ under finite sums and summands.
\end{Not}

Next comes some more specialized notation. We will have cause to deal with both thick subcategories, i.e.\ kernels of exact functors on objects, and certain ideals of morphisms, i.e.\ kernels of additive functors on maps. We thus need notation to distinguish these two types of kernel.

\begin{Not}
Given an additive functor $F\colon \cat{K} \to \cat L$ we set
\begin{displaymath}
\ker F =\{ f\in \cat{K}^{\bullet\to\bullet}\;\vert\; F(f)=0\},
\end{displaymath}
to be the kernel of $F$ on morphisms (\ie the ideal of $F$-phantoms). We will use
\begin{displaymath}
\Ker F = \{X\in \cat{K}\;\vert\; F(X)\cong 0\},
\end{displaymath}
note the uppercase ``K'', for the kernel on objects. We denote by $\Img(F)$ the essential image of~$F$.
\end{Not}

We thus have two flavours of quotient at our disposal: the additive quotient by an ideal of maps, and the Verdier quotient by a thick subcategory of objects (which is really a localization). We introduce notation to distinguish between these two operations.

\begin{Not}
If $\cat{K}$ is a small additive category and $\cat{N}$ is a full subcategory closed under sums and summands we can form the ideal $\mcI$ of maps which factor via $\cat{N}$. We will often write $\cat{K}/\cat{N}$ for the additive quotient $\cat{K} / \mcI$.
\end{Not}

\begin{Not}
\label{not:Verd}%
We denote the \emph{Verdier quotient} of a triangulated category~$\cat{K}$ by a thick subcategory~$\cat{J}\subseteq\cat{K}$ (\ie the localization of~$\cat{K}$ with respect to the morphisms whose cone belongs to~$\cat{J}$) by $\cat{K}\Verd\cat{J}$. This non-standard notation replaces the usual `$\cat{K}/\cat{J}$' in order to prevent confusion with the additive quotient. Note that the Verdier quotient factors via the additive quotient\,:
\[
\xymatrix@C=6em{
\cat{K} \ar[r]^-{\textrm{additive quotient}} \ar@/_1em/[rr]_-{\textrm{Verdier quotient}}
& \cat{K}/\cat{J} \ar[r]^-{\exists\,!}
& \cat{K}\Verd\cat{J}
}
\]
thus justifying the $\Verd$ notation.
\end{Not}

\begin{Rec}
\label{rec:exact}%
We recall some basic facts concerning exact categories; we refer the reader to \cite{Buehler10} for a thorough treatment. An \emph{exact} category $\cat{F}$ is an additive category together with a chosen class $\mcS$ of so-called \emph{admissible} short exact sequences $\xymatrix@C=2em{M'\ar@{ >->}[r]^-{i} & M\ar@{->>}[r]^-{p} & M''}$, consisting of admissible mono\-morphisms~$i$ and admissible epimorphisms~$p$, and satisfying the following axioms: $\mcS$ is closed under isomorphism, contains the split exact sequences and consists only of kernel-cokernel sequences; admissible monomorphisms (epimorphisms) are closed under composition; finally the pushouts of admissible monomorphisms along any morphism exist and remain admissible monomorphisms, and dually for pullbacks of admissible epimorphisms.
\end{Rec}

\begin{Rec}
\label{rec:Frobenius}%
The notions of projective and injective objects in an exact category $\cat{F}$ are defined as usual but with respect to the sequences in~$\mcS$. Projective precovers are admissible epimorphisms with projective source, and dually for injective preenvelopes. An exact category~$\cat{F}$ is \emph{Frobenius} if every object has a projective precover and an injective preenvelope and if injectives and projectives coincide. In that case, the additive quotient of $\cat{F}$ by the projective-injectives, usually denoted $\underline{\cat{F}}$, has a natural triangulated structure. The suspension of $X\in \underline{\cat{F}}$ is given by the cokernel of an injective preenvelope, which can be shown to be well-defined and functorial in $\underline{\cat{F}}$. Given an exact sequence $\xymatrix@C=2em{M'\ar@{ >->}[r]^-{i} & M\ar@{->>}[r]^-{p} & M''}$ there is a commutative diagram
\begin{displaymath}
\xymatrix@R=2em{
M'\ar@{ >->}[r]^-{i} \ar@{=}[d] & M\ar@{->>}[r]^-{p} \ar@{-->}[d]^-{\exists} & M'' \ar@{-->}[d]^-\exists \\
M' \ar@{ >->}[r] & E(M') \ar@{->>}[r] & \Sigma M'
}
\end{displaymath}
where $\xymatrix@C=2em{M'\ar@{ >->}[r]^-{} & E(M')}$ is an injective preenvelope. Going along the top row and then down we get $M'\to M \to M'' \to \Sigma M'$ and the distinguished triangles in $\underline{\cat{F}}$ are precisely those isomorphic to sequences of this form. Further details can be found, for instance, in \cite[Chapter I.2]{Happel88}.
\end{Rec}

\goodbreak
\section{Proper classes of triangles}
\label{se:proper-E}%
\medbreak

There has been a proliferation of terminology, notation, and approaches in relative homological algebra dating from the infancy of the subject. Our reference is Beligiannis~\cite{Beligiannis00}\,(\footnote{Beware that, in~\cite{Beligiannis00}, composition of morphisms $f g$ means $g$ \emph{after}~$f$, although the rule is the usual (reversed) one for functors. At the risk of stunning the reader with consistency, we here write both composition of morphisms and composition of functors in the usual right-to-left order.}). As the latter is rather long, we give a brief overview of some of the nomenclature for the reader's convenience. This section is a preparation for Beligiannis's \Cref{thm:Beligiannis} on the triangulation of relative stable categories. We discuss examples at the end of the section, starting with \Cref{def:F-triangle}.

\begin{Rem}
\label{rem:E-ntuition}%
Relative homological algebra in our triangulated category~$\cat{K}$ is indeed \emph{relative} to a given family of distinguished triangles~$\scE$, whose precise properties we enunciate after this short conceptual warm-up. Most importantly, the family~$\scE$ is not required to be closed under rotation of triangles: We do care very much about the order of the maps in the triangles of~$\scE$. One could think of triangles in~$\scE$
\begin{equation}
\label{eq:E-triangle}%
X\otoo{f} Y \otoo{g} Z \otoo{h} \Sigma X
\end{equation}
as `generalized exact sequences', in which $f$ is a `generalized monomorphism', $g$ a `generalized epimorphism' and $h$ a `generalized zero'. We call them respectively
\begin{center}
\emph{$\scE$-monomorphisms} (the $f$'s), \emph{$\scE$-epimorphisms} (the $g$'s) and \emph{$\scE$-phantoms} (the $h$'s).
\end{center}
We denote the ideal of $\scE$-phantoms by~$\Ph_{\scE}=\SET{h}{\textrm{there is a triangle~\eqref{eq:E-triangle} in }\scE}$.
\end{Rem}

\begin{Hyp}
\label{hyp:E}%
We assume that $\scE$ is a family of distinguished triangles in~$\cat{K}$ closed under isomorphisms, suspension (\ie triple rotation), sums and retracts, and that it contains the trivial triangles \mbox{$X\oto{\id}X\to 0\to \Sigma X$} and \mbox{$0\to Y\oto{\id}Y\to 0$}; we further assume that $\scE$-monomorphisms are closed under homotopy pushouts and $\scE$-epimorphisms are closed under homotopy pullbacks. Finally we assume~$\scE$ to be \emph{saturated} (see~\cite[Section~2.2]{Beligiannis00}) in the sense that if the homotopy pull-back of a map~$g$ along an $\scE$-epimorphism is an $\scE$-epimorphism then $g$ is already an $\scE$-epimorphism. (The dual statement with $\scE$-monomorphisms is equivalent.) Explicitly, given a commutative diagram with distinguished rows and columns
\begin{displaymath}
\xymatrix@C=1.5em@R=1.5em{
0 \ar[r] \ar[d] & X \ar@{=}[r] \ar[d] & X \ar[r] \ar[d] & 0 \ar[d] \\
W \ar@{=}[d] \ar[r] & Y \ar[r] \ar[d] & Y' \ar[r] \ar[d] & \Sigma W \ar@{=}[d] \\
W \ar[r] \ar[d] & Z \ar[r]^-{g} \ar[d] & Z' \ar[r] \ar[d] & \Sigma W \ar[d] \\
0 \ar[r] & \Sigma X \ar@{=}[r] & \Sigma X \ar[r] & 0
}
\end{displaymath}
if the third column $X \to Y' \to Z' \to$ and the second row $W \to Y \to Y' \to $ belong to~$\scE$ then the third row $W \to Z \to Z' \to$ also lies in~$\scE$.

In the language of~\cite{Beligiannis00} one says that $\scE$ is a \emph{proper class of triangles}.
\end{Hyp}

\begin{Rem}
\label{rem:saturated}%
An ideal of morphisms~$\mcI$ in~$\cat{K}$ is called \emph{saturated} (with respect to~$\scE$) if $\mcI$ contains every morphism~$k$ such that $kg$ belongs to~$\mcI$ for some $\scE$-epimorphism~$g$; equivalently, if $\mcI$ contains every morphism $k$ such that $fk$ belongs to~$\mcI$ for~$f$ an $\scE$-monomorphism.
For instance, the saturation property for~$\scE$ expressed in \Cref{hyp:E} is precisely asking the ideal $\mcI=\Ph_\scE$ of $\scE$-phantoms to be a saturated ideal. Furthermore, this ideal~$\mcI$ determines the proper class of triangles~$\scE$, as those distinguished triangles whose third map belongs to~$\mcI$ (cf.~\cite[Section~2]{Beligiannis00}).
\end{Rem}

\begin{Exas}
\label{exa:proper-E}%
There are many examples of proper classes of triangles~$\scE$.
\begin{enumerate}[(1)]
\item
\label{it:proper-E-trivial}%
The largest proper class $\Edist$ and the smallest proper class~$\Esplit$
\begin{align*}
& \Edist\,=\{\textrm{all distinguished triangles}\} \qquad \text{and}
\\
& \Esplit=\{\textrm{all split triangles }X\to Y \to Z \oto{0}\Sigma X\}
\end{align*}
have phantoms all morphisms and the zero morphisms, respectively.
\smallbreak
\item
\label{it:proper-homological}%
Let $H\colon \cat{K}^{(\opname)}\to \cat{A}$ be a (co)homological functor to an abelian category and write $H^i=H\circ\Sigma^i$ as usual for $i\in \bbZ$. The class~$\scE_H$ consisting of the triangles~\eqref{eq:E-triangle} such that $H^i(h)=0$ for all~$i\in \bbZ$, \ie such that for all~$i\in\bbZ$ the sequence \mbox{$0\to H^i(X)\otoo{H^i f} H^i(Y) \otoo{H^i g}H^i (Z)\to 0$} is exact in~$\cat{A}$ (or $\cat{A}\op$ for \emph{co}homological) is a proper class of triangles. The ideal of $\scE_H$-phantoms is precisely $\ker(\prod_{i}H^i)$. The trivial classes $\Esplit$ and $\Edist$ are special cases, for $H\colon \cat{K}\hook \cat{A}(\cat{K})$ the embedding into its Freyd envelope~$\cat{A}(\cat{K})$, and for $H=0$ respectively. More generally, any Serre subcategory~$\cat{B}\subseteq\cat{A}(\cat{K})$ yields a homological functor $H\colon\cat{K}\hook \cat{A}(\cat{K})\onto \cat{A}(\cat{K})/\cat{B}$ and all proper classes appear as~$\scE_H$ in this way, for a suitable $\cat{B}$ (the subcategory of $\cat{A}(\cat{K})$ made of images of $\scE$-phantoms). See~\cite[Section~3]{Beligiannis00}.
\smallbreak
\item
\label{it:F-proper}%
Let $F\colon \cat{K}\to \cat{K}'$ be an exact functor of triangulated categories and~$\scE'$ be a proper class in~$\cat{K}'$. Then $F\inv(\scE')$ is a proper class in~$\cat{K}$. For instance, if $\scE'=\scE_{H'}$ for a homological functor $H'\colon \cat{K}'\to \cat{A}$, then $F\inv(\scE')=\scE_H$ where $H=H'\circ F$. The special case of~$\scE'=\Edist(\cat{K}')$, consisting of all distinguished triangles, is boring since $F\inv(\Edist(\cat{K}'))=\Edist(\cat{K})$. The case of split triangles~$\scE'=\Esplit(\cat{K}')$ will be a much better source of examples. In particular, its ideal of phantoms is the ideal $\ker(F)=\SET{h}{F(h)=0}$.
\end{enumerate}
\end{Exas}

\begin{Defs}
\label{def:proj}%
Following the intuition of \Cref{rem:E-ntuition} we adapt standard notions:
\begin{enumerate}[(1)]
\item
\label{it:proj-inj}%
We say an object $P\in \cat{K}$ is \emph{$\scE$-projective} if for any triangle~\eqref{eq:E-triangle} in~$\scE$ the associated sequence of abelian groups
\begin{displaymath}
\xymatrix@C=1.8em{
0\ar[r] & \cat{K}(P,X) \ar[rr]^-{\cat{K}(P,f)} && \cat{K}(P,Y)\ar[rr]^-{\cat{K}(P,g)} && \cat{K}(P,Z)\ar[r] & 0
}
\end{displaymath}
is exact. Dually we say $I$ is \emph{$\scE$-injective} if for any triangle~\eqref{eq:E-triangle} in~$\scE$ the sequence
\begin{displaymath}
\xymatrix@C=1.8em{
0 \ar[r] & \cat{K}(Z,I) \ar[rr]^-{\cat{K}(g,I)} && \cat{K}(Y,I)\ar[rr]^-{\cat{K}(f,I)} && \cat{K}(X,I)\ar[r] & 0
}
\end{displaymath}
is exact. (In both cases, exactness in the middle is clear and exactness on the left for all triangles in~$\scE$ is equivalent to exactness on the right.) We denote by
\[
\proj(\scE) \qquadtext{and} \inj(\scE)
\]
the classes of $\scE$-projective and $\scE$-injective objects respectively. We note that they are closed under sums, summands, suspensions, and isomorphisms.
\smallbreak
\item
We say that~$\cat{K}$ \emph{has enough $\scE$-projectives} (or sometimes that $\scE$ has enough projectives) if every object~$X\in\cat{K}$ admits a \emph{$\proj(\scE)$-precover} (\emph{$\scE$-projective precover}), \ie an $\scE$-projective~$P$ and an $\scE$-epimorphism $P\to X$. Dually, we say that $\cat{K}$ \emph{has enough $\scE$-injectives} if every $X$ admits an \emph{$\inj(\scE)$-preenvelope} (\emph{$\scE$-injective preenvelope})~$X\to I$, \ie an $\scE$-monomorphism to an $\scE$-injective~$I$.
\end{enumerate}
\end{Defs}

\begin{Exa}
\label{ex:trivial-proj-inj}%
In the trivial cases of \Cref{exa:proper-E}\,\eqref{it:proper-E-trivial}, every object is~$\Esplit$-projective and $\Esplit$-injective; only zero is $\Edist$-projective or $\Edist$-injective.
\end{Exa}

Here are some standard facts, mostly with standard proofs.
\begin{Prop}\label{prop:basic-proper}%
Let $\scE$ be a proper class of triangles in~$\cat{K}$.
\begin{enumerate}[\rm(a)]
\item
\label{it:splittings}%
An $\scE$-epimorphism with $\scE$-projective target is a split epimorphism. An $\scE$-monomorphism with $\scE$-injective source is a split monomorphism.
\smallbreak
\item
For every $\scE$-phantom~$h$, we have $\cat{K}(P,h)=0$ for every $\scE$-projective~$P$ and $\cat{K}(h,I)=0$ for every $\scE$-injective~$I$.
\smallbreak
\item
An object $X$ such that $\cat{K}(X,h)=0$ (respectively $\cat{K}(h,X)=0$) for every $\scE$-phantom~$h$ must be $\scE$-projective (respectively $\scE$-injective).
\end{enumerate}
Suppose now that $\cat{K}$ has enough $\scE$-projectives (respectively enough $\scE$-injectives). Then the following converses to the above hold true:
\begin{enumerate}[\rm(a)]
\setcounter{enumi}{3}
\item
An object $X$ such that every $\scE$-epimorphism $Y\to X$ (respectively every $\scE$-monomorphism $X\to Y$) splits must be $\scE$-projective (respectively $\scE$-injective).
\smallbreak
\item
A morphism $h$ such that $\cat{K}(P,h)=0$ for every $\scE$-projective~$P$ (respectively $\cat{K}(h,I)=0$ for every $\scE$-injective~$I$) must be an~$\scE$-phantom.
\end{enumerate}
\end{Prop}

\begin{proof}
Parts~(a), (b), (c) and~(d) are easy exercises from the definition and the long exact sequences associated to distinguished triangles under $\cat{K}(X,-)$ and~$\cat{K}(-,X)$.
Let us prove~(e) when $\cat{K}$ has enough projectives, say. Choose an $\scE$-projective precover $g\colon P\to Z$ of the source of the given morphism $h\colon Z\to W$. Then in~$\cat{K}(P,W)$, we have $hg=\cat{K}(P,h)(g)=0$. Since $\scE$-phantoms are saturated (\Cref{rem:saturated}) we deduce from~$hg=0\in\Ph_{\scE}$ with $g$ an $\scE$-epimorphism that $h$ is an $\scE$-phantom.
\end{proof}

\begin{Def}
\label{def:relative-exact}%
Let $F\colon \cat{K}\to \cat{K}'$ be an exact functor of triangulated categories. Suppose that $\scE$ and $\scE'$ are proper classes of triangles in~$\cat{K}$ and~$\cat{K}'$ respectively. We say that $F$ is \emph{relative-exact (with respect to~$\scE$ and~$\scE'$)} if~$F(\scE)\subseteq\scE'$.
\end{Def}

We encounter an old friend:
\begin{Prop}
\label{prop:adj-exact}%
Let $F\colon \cat{K}\to \cat{K}'$ be relative-exact with respect to~$\scE$ in~$\cat{K}$ and $\scE'$ in~$\cat{K}'$. A right adjoint~$U\colon \cat{K}'\to \cat{K}$ of~$F$, if it exists, must preserve injectives\,: $U(\inj(\scE'))\subseteq\inj(\scE)$. Dually, a left adjoint~$V\colon \cat{K}'\to \cat{K}$ of~$F$, if it exists, must preserve projectives\,: $V(\proj(\scE'))\subseteq\proj(\scE)$.
\end{Prop}

\begin{proof}
For $I\in\inj(\scE')$, the object $U(I)$ satisfies \Cref{def:proj}\,\eqref{it:proj-inj} since $\cat{K}(-,U(I))$ $\cong \cat{K}'(F(-),I)=\cat{K}'(-,I)\circ F$ and $F$ is relative-exact and $I$ is injective.
\end{proof}

\threestars

We conclude this section by discussing a method to produce proper classes of triangles by returning to the `$F$-split' triangles~$\scE=F\inv(\Esplit)$ of \Cref{exa:proper-E}\,\eqref{it:F-proper}. We begin by lightening the terminology in the natural way:

\begin{Def}
\label{def:F-triangle}%
Let $F\colon\cat{K}\to\cat{L}$ be an exact functor of triangulated categories. We say that a triangle
\begin{equation}
\label{eq:F-triangle}%
X\otoo{f} Y \otoo{g} Z \otoo{h} \Sigma X
\end{equation}
is \emph{$F$-split} if $F(h)=0$, in which case we say that $f$ is an \emph{$F$-monomorphism} (\ie $F(f)$ is a split monomorphism), that $g$ is an \emph{$F$-epimorphism} (\ie $F(g)$ is a split epimorphism) and that $h$ is an \emph{$F$-phantom}. By \Cref{exa:proper-E}\,\eqref{it:F-proper}, the class $\scE_F=F\inv(\Esplit)$ of $F$-split triangles is a proper class of triangles. The corresponding ideal $\Ph_{\scE_F}$ of $F$-phantoms is simply the kernel~$\ker(F)$ of $F$ on morphisms. The $\scE_F$-projective objects and $\scE_F$-injective objects will simply be called \emph{$F$-projective} and \emph{$F$-injective} respectively.
\end{Def}

The definitions of $F$-split triangles, $F$-monomorphisms, $F$-epimorphisms and $F$-phantoms are clear and transparently controlled by~$F$. On the other hand, $F$-injectives and $F$-projectives are more mysterious, and it is not so obvious how to guarantee the existence of enough of them. This is our starting point.

\begin{Lem}\label{lem:adjunction}
Let $F\colon \cat{K} \to \cat L$ be an exact functor.
\begin{enumerate}[\rm(a)]
\item
Suppose that $F$ admits a right adjoint $U\colon\cat{L}\to\cat{K}$. Then $U(L)$ is $F$-injective for every $L\in \cat L$ and the unit of adjunction $\eta\colon X \to UF(X)$ is an $F$-monomorphism for every $X\in \cat{K}$. In particular, there are enough $F$-injectives and
\begin{displaymath}
F\text{-}\inj = \add(U(L)\;\vert\; L\in \cat L).
\end{displaymath}
\item
Suppose that $F$ admits a left adjoint $V\colon\cat{L}\to\cat{K}$. Then $V(L)$ is $F$-projective for every $L\in \cat L$ and the counit of adjunction $\varepsilon\colon VF(X) \to X$ is an $F$-epimorphism for every $X\in \cat{K}$. In particular, there are enough $F$-projectives and
\begin{displaymath}
F\text{-}\proj = \add(V(L)\;\vert\; L\in \cat L).
\end{displaymath}
\end{enumerate}
\end{Lem}

\begin{proof}
This is standard. By construction of~$\scE_F=F\inv(\Esplit)$, the exact functor $F$ is relative-exact (\Cref{def:relative-exact}) with respect to~$\scE_F$ in~$\cat{K}$ and $\Esplit$ in~$\cat{L}$. By \Cref{prop:adj-exact}, $U$ and $V$ preserve injectives and projectives, respectively. But every object of~$\cat{L}$ is~$\Esplit$-projective-injective (\Cref{ex:trivial-proj-inj}). The claims about the (co)units of adjunction are immediate from the unit-counit relations.
\end{proof}

\begin{Rem}
One can develop relative homological algebra via resolutions, derived functors, etc, as in~\cite[Section~4]{Beligiannis00} or~Christensen~\cite{Christensen98}. See also~\cite{MeyerNest10,Meyer08b} for an illustration of these ideas in $KK$-theory. We bifurcate instead towards `singularity categories', where projective and injective objects vanish.
\end{Rem}

\goodbreak
\section{Beligiannis' construction}
\label{se:Beligiannis}%
\medbreak

In usual homological algebra, the most convenient setting for constructing a stable category is a Frobenius exact category (see~\Cref{rec:exact}). The transposition of these ideas to relative homological algebra is the subject of this section.

Let $\cat{K}$ be an essentially small triangulated category and~$\scE$ be a proper class of triangles as in \Cref{se:proper-E}. We further assume the following strong symmetry in~$\scE$:

\begin{Def}
\label{def:Frobenius}%
We say that a proper class of triangles $\scE$ as in \Cref{hyp:E} is \emph{Frobenius} if it has enough projectives and enough injectives and if they coincide
\begin{displaymath}
\proj(\scE) = \inj(\scE).
\end{displaymath}
We refer to these as the $\scE$-projective-injective objects, when we want to stress Frobeniusity. Otherwise, we simply speak of $\scE$-projectives or $\scE$-injectives, as needed.
\end{Def}

\begin{Def}
\label{def:rel-stable}%
For $\scE$ Frobenius, we consider the \emph{$\scE$-relative stable category}, \ie the corresponding additive quotient of~$\cat{K}$ by the $\scE$-projective-injective objects
\[
\STAB{\cat{K}}{\scE}=\frac{\cat{K}}{\proj(\scE)}=\frac{\cat{K}}{\inj(\scE)}\,.
\]
This quotient comes equipped with the universal additive functor
\[
\pi\colon \cat{K}\to \SKE
\]
mapping $\proj(\scE)=\inj(\scE)$ to zero. Recall that the category $\SKE$ can be realized by keeping the same objects as~$\cat{K}$ (so that $\pi$ is the identity on objects) and by defining the hom groups as the quotients of abelian groups
\[
\STAB{\cat{K}}{\scE}(X,Y)=\frac{\cat{K}(X,Y)}{\NE(X,Y)}
\]
where the ideal of morphisms~$\NE$ is defined for every pair of objects~$X,Y$ as
\[
\NE(X,Y)=\SET{f\colon X\to Y}{f\textrm{ factors via an $\scE$-projective-injective object}}.
\]
We shall call a morphism \emph{$\scE$-contractible} if it belongs to~$\NE$.
\end{Def}

\begin{War}
It is important not to confuse the two ideals we discuss here: The ideal $\NE$ of $\scE$-contractible morphisms (those factoring via $\scE$-projectives-injectives) and the ideal~$\Ph_{\scE}$ of $\scE$-phantoms (those appearing as the third map in an $\scE$-triangle).
\end{War}

\begin{Cons}
\label{cons:stable-triang}%
The idea is to put a triangulated structure on $\SKE$ by working in~$\cat{K}$ but relative to $\scE$, in the same way that one triangulates the stable category of a Frobenius exact category; suspension and cones in $\SKE$ will be given by the `same' homotopy pushouts/pullbacks as usual but where we think of $\proj(\scE)$-precovers and $\inj(\scE)$-preenvelopes as being the correct maps from or to a contractible object. Note for peace of mind that $\pi(X)\cong0$ if and only if $X$ is projective-injective (\Cref{lem:iso}).

Given an object $X\in \SKE$ we define $\Sigma_{\scE} X$ to be the image under $\pi$ of the cofibre in~$\cat{K}$ of a chosen $\inj(\scE)$-preenvelope $X \to X_\scE$. One can show that the result is independent of the choice of the preenvelope, up to a canonical isomorphism in~$\SKE$. Similarly $\Sigma_{\scE}^{-1}$ is defined to be $\pi$ applied to the fibre of any chosen $\proj(\scE)$-precover $X^\scE \to X$. Both of these constructions are well-defined and functorial in $\SKE$ and they are quasi-inverse to one another. This, as the notation suggests, defines the suspension on~$\SKE$. Further details can be found in Construction~\ref{cons:sigma-comparison}.

Given a morphism $f\in \SKE(X,Y)$ we now construct its cofibre in~$\SKE$. Let $\tilde{f} \in \cat{K}(X,Y)$ be a lift of $f$ to $\cat{K}$, choose an $\inj(\scE)$-preenvelope $X\to X_\scE$ of $X$ and consider the homotopy pushout in~$\cat{K}$, as in the leftmost square of
\begin{equation}
\label{eq:stable-triangle}%
\vcenter{
\xymatrix{
X \ar[r] \ar[d]_-{\tilde{f}} & X_\scE \ar[d] \ar[r] & \Sigma_{\scE}X \ar@{=}[d] \\
Y \ar[r]_-{\tilde{g}} & Z \ar[r]_-{\tilde{h}} & \Sigma_{\scE}X
}}
\end{equation}
We define the cofibre of $f$ to be $\pi(Y \stackrel{\tilde{g}}{\to} Z)$. To obtain the distinguished triangle in $\SKE$ these maps fit into, we complete the homotopy bicartesian square by taking cones (in~$\cat{K}$) to get the rest of~\eqref{eq:stable-triangle} above and we declare the diagram
\[
X\otoo{f} Y \otoo{\pi\tilde g} Z \otoo{\pi\tilde h} \Sigma_{\scE}X
\]
to be a `basic' distinguished triangle in~$\SKE$. More generally, any triangle isomorphic in~$\SKE$ to a basic one is distinguished. We can define the fibre of $f$ by taking the analogous homotopy pullback along a $\proj(\scE)$-precover of~$Y$. And we can construct distinguished triangles in an apparently dual fashion, that actually coincides with the first one.
\end{Cons}

\begin{Exa}
\label{ex:st-triangle-E}%
Let us discuss the cofibre in~$\SKE$ of the class of an $\scE$-monomorphism. Since every morphism in~$\SKE$ is isomorphic to the class of an $\scE$-monomorphism (by replacing $f\colon X\to Y$ with $\smat{i\\f}\colon X\to X_{\scE}\oplus Y$ where $i\colon X\to X_{\scE}$ is an $\inj(\scE)$-preenvelope), this example is actually generic. Consider a triangle in~$\scE$
\begin{equation}
\label{eq:E-triangle-again}%
X\otoo{f} Y \otoo{g} Z \otoo{h} \Sigma X
\end{equation}
and the homotopy pushout~$W$ of~$X\oto{f} Y$ along an $\inj(\scE)$-preenvelope $X\to X_{\scE}$. Since the morphism $X_{\scE}\to W$ is an $\scE$-monomorphism (homotopy pushout of the $\scE$-monomorphism~$f$) and since $X_{\scE}$ is chosen to be $\scE$-injective, the morphism \mbox{$X_{\scE}\to W$} splits by \Cref{prop:basic-proper}\,\eqref{it:splittings}. We therefore obtain a commutative diagram with exact rows and columns in~$\cat{K}$:
\[
\xymatrix@C=5em{
X \ar[r] \ar[d]_-{f}
& X_\scE \ar[d] \ar[r] \ar[d]^-{\smat{1\\0}}
& \Sigma_{\scE}X \ar@{=}[d]
\\
Y \ar[r]_-{\tilde{g}=\smat{\tilde g_1\\g}} \ar[d]_-{g}
& X_{\scE}\oplus Z \ar[r]_-{\tilde{h}=\smat{\tilde h_1 & \tilde h_2}} \ar[d]_-{\smat{0&1}}
& \Sigma_{\scE}X
\\
Z \ar@{=}[r]
& Z
}
\]
As the morphism $\smat{0&1}\colon X_{\scE}\oplus Z\to Z$ becomes an isomorphism in~$\SKE$, with explicit inverse~$\smat{0\\1}\colon Z\to X_{\scE}\oplus Z$, we see that the distinguished triangle of \Cref{cons:stable-triang} over the image~$\pi(f)\colon X\to Y$ in~$\SKE$ is isomorphic to the triangle
\begin{equation}
\label{eq:image-triangle}%
\xymatrix{
X \ar[r]^-{\pi f}
& Y \ar[r]^-{\pi g}
& Z \ar[r]^-{\pi\tilde h_2}
& \Sigma_{\scE} X
}
\end{equation}
which is \emph{almost} the image of the initial triangle~\eqref{eq:E-triangle-again}. Only the third morphism differs, for the simple reason that its target is not the same as that of~$\pi h$, since the suspensions differ in~$\cat{K}$ and~$\SKE$. We shall nevertheless call~\eqref{eq:image-triangle} the \emph{image-triangle} of the $\scE$-triangle~\eqref{eq:E-triangle-again}. Again, we shall give further details in \Cref{cons:sigma-comparison}.
\end{Exa}

The first stepping stone for our study is the following very general result:

\begin{Thm}[Beligiannis]
\label{thm:Beligiannis}%
Let $\scE$ be a Frobenius class of triangles in~$\cat{K}$, as in \Cref{def:Frobenius}. \Cref{cons:stable-triang} makes the $\scE$-relative stable category~$\SKE$ into a triangulated category. Moreover, for every triangle in~$\scE$ of the form~\eqref{eq:E-triangle-again}, the image-triangle~\eqref{eq:image-triangle} is a distinguished triangle in $\SKE$ and every distinguished triangle in $\SKE$ is isomorphic, in $\SKE$, to a triangle of this form.
\end{Thm}

\begin{Rem}
The above result is given in \cite[Theorem~7.2]{Beligiannis00}, with further reference to~\cite[Theorem~2.12]{BeligiannisMarmaridis94}. Note that the latter makes a drastic assumption at some point, which works fine in models but not in full generality, in that kernels of some maps should exist. This assumption is not formally needed, as the careful reader will verify.
\end{Rem}

\begin{Exa}
The trivial classes of triangles $\Esplit$ and $\Edist$ of \Cref{exa:proper-E}\,\eqref{it:proper-E-trivial} give us the trivial quotients $\STAB{\cat{K}}{\Esplit}=0$ and $\STAB{\cat{K}}{\Edist}=\cat{K}$ respectively.
\end{Exa}

\begin{War}
The quotient $\pi\colon\cat{K}\to \SKE$ is \emph{not} necessarily exact, and in fact it will usually not be, as it does not even respect the suspension. We can say more:
\end{War}

\begin{Lem}\label{lem:notexact1}
For~$\scE$ Frobenius, the additive quotient functor $\pi\colon \cat{K}\to \SKE$ is exact if and only if it is a localization. In that case, and if~$\cat{K}$ is idempotent-complete, then $\pi$ is the projection of~$\cat{K}$ onto a direct summand, \ie $\cat{K}=\cat{J}\oplus\cat{J}'$ where $\cat{J}=\proj(\scE)$ and $\cat{J}'\cong\SKE$ are orthogonal subcategories of~$\cat{K}$, both ways: 
$\cat{K}(\cat{J}, \cat{J}') = 0 = \cat{K}(\cat{J}',\cat{J})$.
\end{Lem}

\begin{proof}
Localizations are exact so the sufficiency is clear. Conversely, suppose that $\pi$ is exact. Then $\cat{J}:=\Ker \pi = \proj(\scE)$ is thick, and we get a commutative diagram
\begin{displaymath}
\xymatrix@C=2em@R=1em{
& \cat{K} \ar[rd]^-\pi \ar[ld]
\\
\cat{K}\Verd\proj(\scE) \ar@{-->}[rr]<-0.5ex> && \SKE \ar@{-->}[ll]<-0.5ex>
}
\end{displaymath}
where the dashed arrows exist by the respective universal properties and are both exact (one by the universal property and one by inspection). It follows from the uniqueness attached to the universal properties that these functors give an exact equivalence $\cat{K}\Verd\proj(\scE)\cong \SKE$, identifying $\SKE$ with the Verdier quotient.

Now, suppose that this holds and furthermore that~$\cat{K}$ is idempotent-complete. Then for every~$X\in\cat{K}$, the distinguished triangle~$Z\oto{f} X\oto{i}X_{\scE}\to$ completing a preenvelope $X\to X_{\scE}$ maps to~$\pi(Z)\isoto \pi(X) \to 0\to $ in~$\SKE$ and it follows by fullness of~$\pi$ that there exists a morphism $g\colon X\to Z$ in~$\cat{K}$ such that $\pi(f g)=\pi(\id_X)$. Hence there exists $h\colon X_{\scE} \to X$ such that $\id_X=f g+h i$. Applying~$i$ to this relation and using~$i f=0$, we see that $i=i h i$. It follows that $h i\colon X\to X$ and $i h\colon X_{\scE}\to X_{\scE}$ are idempotents. Since $\cat{K}$ is idempotent-complete, the triangle is split, namely, $X\cong I_0\oplus Z_0$, $X_{\scE}=I_0\oplus I_1$ and $Z=\Sigma\inv I_1\oplus Z_0$ so that the triangle becomes
\[
\xymatrix{
Z=\Sigma\inv I_1 \oplus Z_0 \ar[r]^-{f=\smat{0&0\\0&1}}
& X=I_0 \oplus Z_0 \ar[r]^-{i=\smat{1&0\\0&0}}
& X_{\scE}=I_0 \oplus I_1 \ar[r]^-{\smat{0&1\\0&0}}
& \Sigma Z=I_1\oplus \Sigma Z_0\,.
}
\]
Since $i\colon X\to X_{\scE}$ is a pre-envelope, and yet is zero on~$Z_0$, it is easy to show that $Z_0\in {}^\perp\cat{J}$. So we have proved that every object of~$\cat{K}$ is the direct sum of an object of~$\cat{J}$ and of~${}^\perp\cat{J}$. Dually, one can show that $\cat{K}=\cat{J}\oplus \cat{J}^\perp$. From this, it follows easily that ${}^\perp\cat{J}=\cat{J}^\perp$ and that this subcategory is equivalent to~$\SKE$ under~$\pi$.
\end{proof}

The assumption that~$\cat{K}$ is idempotent-complete is necessary in \Cref{lem:notexact1}.

\begin{Exa}
\label{exa:non-split}%
Suppose that $\cat{L}$ is a triangulated category with non-zero Grothendieck group $K_0(\cat{L})\neq 0$. Let $\cat{K}\subset\cat{L}\oplus\cat{L}$ be the triangulated subcategory on the following objects $\cat{K}=\SET{(X,Y)\in\cat{L}\times\cat{L}}{[X]=[Y]\textrm{ in }K_0(\cat{L})}$. Let~$\pi\colon \cat{K}\to \cat{L}$ be the projection onto the first summand. It is clearly exact and full. Let $\cat{J}=\Ker(\pi)=\SET{(0,Y)}{[Y]=0}$. It is easy to check that $\pi\colon \cat{K}\to\cat{L}$ realizes the additive quotient of~$\cat{K}$ by~$\cat{J}$, using that for every $Y\in \cat{L}$ the object $(0,Y\oplus \Sigma Y)$ belongs to~$\cat{J}$. The same trick proves that the conclusion of \Cref{lem:notexact1} does not hold. Indeed, it shows that $\cat{J}^\perp={}^\perp\cat{J}=\SET{(X,0)}{[X]=0}$. Therefore $\cat{J}\oplus \cat{J}^\perp=\SET{(X,Y)}{[X]=[Y]=0}$ is strictly smaller than~$\cat{K}$ since we assume $K_0(\cat{L})\neq0$.
\end{Exa}

Let us say a word on the naturality of the construction $(\cat{K},\scE)\mapsto \SKE$.

\begin{Prop}\label{prop:naturality}
Let $F\colon \cat{K}\to \cat{K}'$ be an exact functor and suppose that $\scE$ and $\scE'$ are Frobenius classes of triangles in $\cat{K}$ and $\cat{K}'$, with corresponding additive quotients~$\SKE$ and~$\SKEp$ respectively. Provided we have
\begin{enumerate}[\rm(1)]
\item
\label{it:naturality-1}%
$F$ is relative-exact (\Cref{def:relative-exact}), that is, $F(\scE) \subseteq \scE'$
\item
\label{it:naturality-2}%
$F$ preserves projective-injectives, that is, $F(\proj \scE) \subseteq \proj \scE'$
\end{enumerate}
then there is a unique exact functor $\underline{F}\colon \SKE\to \SKEp$ such that the square
\begin{displaymath}
\xymatrix@R=1.5em{
\cat{K} \ar[r]^-F \ar[d]_-\pi & \cat{K}' \ar[d]^-{\pi'} \\
\SKE \ar[r]_-{\underline{F}} & \SKEp
}
\end{displaymath}
commutes. This construction is natural in~$F$ in the obvious sense (without restriction on the allowed natural transformations).
\end{Prop}

\begin{proof}
This is an easy familiarizing exercise. By (2) the functor $F$ induces a unique functor $\underline{F}$ making the desired square commute. Exactness is direct from \Cref{cons:stable-triang} (or \Cref{ex:st-triangle-E}) using~(1).
\end{proof}

\begin{Rem}
\label{rem:preserve-proj}%
By \Cref{prop:adj-exact}, if $F$ admits a left (or right) adjoint which is relative-exact, then $F$ automatically preserves projectives (injectives).
\end{Rem}

\begin{Rem}
If two functors $F\colon \cat{K}\to \cat{K}'$ and $G\colon \cat{K}'\to\cat{K}''$ satisfy the assumptions of \Cref{prop:naturality} then so does $G\circ F$ and we have $\underline{G\circ F}=\underline{G}\circ\underline{F}$.
\end{Rem}

An easy but useful consequence of this functoriality is that one can very cheaply obtain exact autoequivalences of the additive quotient.

\begin{Cor}\label{cor:auto}
Let $\scE$ be a Frobenius class of triangles in~$\cat{K}$. If $F$ is an exact autoequivalence of $\cat{K}$ such that $F(\scE) = \scE$ (or equivalently $F(\proj \scE) = \proj \scE$) then $F$ induces a unique exact autoequivalence $\underline{F}$ of~$\SKE$ such that $\underline{F}\,\pi=\pi\,F$.
\end{Cor}

\begin{proof}
Use that $F$ has an adjoint quasi-inverse~$F\inv$ and the above two remarks.
\end{proof}

\begin{Exa}\label{exa:auto}
Any proper class of triangles~$\scE$ is closed under suspension~$\Sigma$ and hence $\Sigma\colon\cat{K}\to \cat{K}$ induces a (skew-)\,exact autoequivalence $\underline{\Sigma}\colon \SKE\isoto \SKE$. However, this induced autoequivalence is, in general, \emph{not} the suspension~$\Sigma_{\scE}$ on~$\SKE$. Although $\Sigma_{\scE}$ and $\underline{\Sigma}$ need not agree there is a natural comparison map between the two.
\end{Exa}

\begin{Cons}\label{cons:sigma-comparison}
In order to define $\Sigma_{\scE}$ we have fixed for each $X\in \cat{K}$ a triangle
\begin{displaymath}
X\otoo{i_X} X_{\scE} \otoo{p_X} X' \otoo{\alpha_X} \Sigma X,
\end{displaymath}
with $i_X$ an $\inj(\scE)$-preenvelope, and declared $\Sigma_{\scE}(\pi(X)) = \pi(X')$. On the other hand, $\underline{\Sigma}(\pi(X))$ is simply~$\pi(\Sigma(X))$. We can thus define a transformation $\underline{\alpha}\colon \Sigma_{\scE} \to \underline{\Sigma}$ by taking its component at each object $\pi(X)$ to be
\begin{displaymath}
\xymatrix{
\Sigma_{\scE}\pi(X) = \pi(X') \ar[rr]^-{\pi(\alpha_X)} && \pi(\Sigma X) = \underline{\Sigma}\pi(X).
}
\end{displaymath}
Of course $\underline{\alpha}$ depends on the choices we have made in producing $\Sigma_{\scE}$, but making different choices gives a natural transformation which is object-wise conjugate, by canonical isomorphisms, to $\underline{\alpha}$.
\end{Cons}

\begin{Lem}\label{lem:sigma-comparison}
The transformation $\underline{\alpha}\colon \Sigma_{\scE} \to \underline{\Sigma}$ constructed above is natural and graded, in the sense that the natural map $t\colon \Sigma_{\scE} \circ \underline{\Sigma} \to \underline{\Sigma}\circ \Sigma_{\scE}$, that is part of the structure of $\underline{\Sigma}$ being an exact functor, makes the square
\begin{displaymath}
\xymatrix{
\Sigma_{\scE} \circ \underline{\Sigma} \ar[r]^-t \ar[d]_-{\underline{\alpha}_{\underline{\Sigma}}} & \underline{\Sigma}\circ \Sigma_{\scE} \ar[d]^-{\underline{\Sigma}(\underline{\alpha})} \\
\Sigma^2 \ar[r]_-{-1} & \Sigma^2
}
\end{displaymath}
commute.
\end{Lem}

\begin{proof}
Let $f\colon X\to Y$ be a map in $\cat{K}$ representing a class $\pi(f)\colon \pi(X) \to \pi(Y)$ in~$\SKE$. In $\cat{K}$ we have a commutative diagram
\begin{displaymath}
\xymatrix{
X \ar[r] \ar[d]_-f & X_{\scE} \ar[r] \ar@{-->}[d]^-{\exists} & X' \ar[r] \ar@{..>}[d]^-{\exists f'} & \Sigma X \ar[d]^-{\Sigma f} \\
Y \ar[r] & Y_{\scE} \ar[r] & Y' \ar[r] & \Sigma Y
}
\end{displaymath}
with $\pi(f') = \Sigma_{\scE}\pi(f)$ by definition. Applying $\pi$ to the righthand square we deduce a commutative diagram exhibiting the naturality of $\underline{\alpha}$:
\begin{displaymath}
\xymatrix{
\Sigma_{\scE} \pi(X) \ar@{=}[r] \ar[d]^-{\Sigma_{\scE}\pi(f)} & \pi(X') \ar[r]^-{\underline{\alpha}_{\pi(X)}} \ar[d]^-{\pi(f')} & \pi \Sigma X \ar[d]^-{\pi(\Sigma f)} \ar@{=}[r] & \underline{\Sigma} \pi X \ar[d]^-{\underline{\Sigma}\pi(f)} \\
\Sigma_{\scE} \pi(Y) \ar@{=}[r] & \pi(Y') \ar[r]^-{\underline{\alpha}_{\pi(Y)}} & \pi \Sigma Y \ar@{=}[r] & \underline{\Sigma} \pi Y
}
\end{displaymath}
For the statement that $\underline{\alpha}$ is graded we construct $t\colon \Sigma_{\scE} \circ \underline{\Sigma} \to \underline{\Sigma}\circ \Sigma_{\scE}$. Consider the diagram in $\cK$
\begin{displaymath}
\xymatrix{
\Sigma X \ar[r]^-{i_{\Sigma X}} \ar@{=}[d] & (\Sigma X)_{\scE} \ar[r]^-{p_{\Sigma X}} \ar@{-->}[d]^-{\exists} & (\Sigma X)' \ar[r]^-{\alpha_{\Sigma X}} \ar@{..>}[d]^-{\exists \widetilde{t}_X} & \Sigma^2 X \ar@{=}[d] \\
\Sigma X \ar[r]_-{-\Sigma i_X} & \Sigma(X_{\scE}) \ar[r]_-{-\Sigma p_X} & \Sigma(X') \ar[r]_-{-\Sigma \alpha_X} & \Sigma^2 X
}
\end{displaymath}
where we have produced the dashed and then dotted arrow via the fact that $-\Sigma i_X$ is an $\inj\scE$ preenvelope and then the usual axiom. The component of the natural transformation $t\colon \Sigma_{\scE} \circ \underline{\Sigma} \to \underline{\Sigma}\circ \Sigma_{\scE}$ at $\pi(X)$ is given by $\pi(\widetilde{t}_X)$. Applying $\pi$ to the rightmost square gives a commutative diagram in $\SKE$
\begin{displaymath}
\xymatrix{
\Sigma_{\scE} \underline{\Sigma} \pi(X) \ar[rr]^-{\underline{\alpha}_{\underline{\Sigma} \pi(X)}} \ar[d]_-{t_{\pi(X)}}  && \underline{\Sigma}^2 \pi(X) \ar@{=}[d] \\
\underline{\Sigma} \Sigma_{\scE} \pi(X) \ar[rr]_-{-\underline{\Sigma} (\underline{\alpha}_{\pi(X)})} && \underline{\Sigma}^2 \pi(X)
}
\end{displaymath}
as required.
\end{proof}

Here is a naturality statement for adjoint pairs, based on Proposition~\ref{prop:naturality}, with a somewhat slicker formulation.
\begin{Lem}\label{lem:adjoints-descend}
Let $F\colon \cat{K}\to\cat{K}'$ be an exact functor with right adjoint~$G\colon \cat{K}'\to \cat{K}$. Suppose that $\scE$ and $\scE'$ are Frobenius classes of triangles on~$\cat{K}$ and~$\cat{K}'$ respectively, and that $F$ and $G$ are relative-exact with respect to those. Then necessarily
\begin{displaymath}
F(\proj(\scE)) \subseteq \proj(\scE')
\qquadtext{and}
G(\proj(\scE')) \subseteq \proj(\scE)
\end{displaymath}
and $F$ and $G$ descend to an adjoint pair of (exact) functors $\underline{F}\colon \SKE \adjto \SKEp : \underline{G}$.
\end{Lem}

\begin{proof}
The conservation of projective-injectives follows from \Cref{prop:adj-exact}. By \Cref{prop:naturality}, $F$ and $G$ induce exact functors on $\SKE$ and~$\SKEp$ and the unit and counits of the adjunction descend as well by naturality. The unit-counit relations are then simply the images of the unit-counit relations in~$\cat{K}$ and~$\cat{K}'$ under~$\pi$ and~$\pi'$.
\end{proof}

We finish the section by discussing when the proper classes~$\scE_F$ of $F$-split triangles (\Cref{def:F-triangle}) is Frobenius.

\begin{Thm}\label{thm:Wirthmuller}
Suppose that we are given an exact functor $F\colon \cat{K} \to \cat L$ with both a left and a right adjoint
\begin{displaymath}
\xymatrix{
\cat{K}
 \ar[d]|-{F\vphantom{I^I_J}}
\\
\cat{L}
 \ar@<-1em>[u]_-{U}
 \ar@<1em>[u]^-{V}
}
\end{displaymath}
satisfying the following generalized Wirthm\"uller condition: There exists an equivalence $W\colon \cat L \to \cat L$ such that $V\circ W \simeq U$. Then the $F$-injectives and $F$-projectives coincide in~$\cat{K}$ and are given by
\begin{displaymath}
\add(\Img(U)) = \add(\Img(V)).
\end{displaymath}
In other words, the class $\scE_F=F\inv(\Esplit)$ of $F$-split triangles is Frobenius. In particular, there is an induced triangulated structure on~$\SKF:=\STAB{\cat{K}}{\scE_F}$ via \Cref{thm:Beligiannis}.
\end{Thm}

\begin{proof}
Since $U \simeq V\circ W$ we have, for $X\in \cat{K}$, that $X\simeq U(L)$ for some $L\in \cat L$ if and only if $X\simeq VW(L)$ (for the same $L$) and so the essential images of $U$ and $V$ coincide. Hence $\add(\Img(U)) = \add(\Img(V))$, and we conclude by \Cref{lem:adjunction}.
\end{proof}

\begin{Rem}
\label{rem:Wirth}%
The `Wirthm\"ullery' condition in the theorem is reminiscent of the various generalizations of Frobenius functor and plays a similar role in producing situations well adapted to relative homological algebra (cf.\ for instance~\cite[Section~3]{DellAmbrogioStevensonStovicek17}). It also generalizes the notion of \emph{spherical functor} (we direct the uninitiated but curious reader to \cite[Section~2]{Addington16}) and so we immediately obtain many examples from algebraic geometry.

Another important case in which the theorem applies is when the two adjoints are equal: $V = U$. This situation is often referred to as a \emph{Frobenius adjunction} $V\adj F\adj V$, see~\cite{Morita65}. This occurs when one considers induction, restriction, and coinduction along the inclusion of a subgroup of a finite group, or when one considers tensoring with a rigid object.
\end{Rem}

\goodbreak
\section{Thick subcategories and proto-birationality}
\label{se:thick}%
\medbreak

We discuss the thick subcategories of the relative stable category~$\SKE$.

\begin{Hyp}
\label{hyp:se:thick}%
Let $\cat{K}$ be triangulated, $\scE$ a Frobenius class of triangles and
\[
\pi\colon\cat{K}\to\SKE
\]
the associated relative stable category as in \Cref{se:Beligiannis}.
\end{Hyp}

Let us begin with the question of stable isomorphisms.

\begin{Lem}\label{lem:iso}
Under \Cref{hyp:se:thick}, two objects $X,Y\in\cat{K}$ have isomorphic images $\pi(X)\simeq\pi(Y)$ in~$\SKE$ if and only if there are $\scE$-projective-injective objects $C_1$ and $C_2$ and an isomorphism $X\oplus C_1 \simeq Y \oplus C_2$ in~$\cat{K}$.
\end{Lem}

\begin{proof}
Pick an isomorphism $\pi f\colon \pi X \to \pi Y$ and pick a $\proj(\scE)$-precover $p\colon C_1\to Y$. Complete the morphism $(f\ p)\colon X\oplus C_1\to Y$ into a distinguished triangle in~$\cat{K}$
\begin{displaymath}
\xymatrix{
W \ar[r]^-{w} & X\oplus C_1 \ar[r]^-{\left( \begin{smallmatrix} f & p \end{smallmatrix} \right) } & Y \ar[r] & \Sigma W.
}
\end{displaymath}
Since the second map is an $\scE$-epimorphism, this is an $\scE$-triangle. By \Cref{ex:st-triangle-E}, its image-triangle is distinguished in~$\SKE$ and since $\pi(f\ p)\cong\pi(f)$ is an isomorphism, we have $\pi(W)=0$. So $W$ is $\scE$-injective and the $\scE$-monomorphism $w$ splits by \Cref{prop:basic-proper}\,\eqref{it:splittings}, yielding the desired relation $Y\oplus W\simeq X\oplus C_1$ in~$\cat{K}$.
\end{proof}

\begin{Rem}
The above \Cref{lem:iso} justifies the name `stable category relative to~$\scE$': It is the place where objects become isomorphic if they are isomorphic in~$\cat{K}$ up to adding summands in the subcategory~$\proj(\scE)=\inj(\scE)$.
\end{Rem}

We now express the thick subcategories of~$\SKE$ in terms of subcategories of~$\cat{K}$.

\begin{Prop}
\label{prop:thick-rel-stable}%
Under \Cref{hyp:se:thick}, there is an order-preserving bijection $\cat{J}\mapsto \pi\inv(\cat{J})$, with inverse $\cat{C}\mapsto \pi(\cat{C})$, between the thick subcategories $\cat{J}\subseteq\SKE$ and the additive subcategories~$\cat{C}\subseteq\cat{K}$ containing~$\proj(\scE)=\inj(\scE)$, closed under direct summands and satisfying the following \emph{$\scE$-two-out-of-three property}\,: For every $\scE$-triangle $X\to Y\to Z\to \Sigma X$ in~$\cat{K}$, if two out of~$X,\ Y,\ Z$ belong to~$\cat{C}$ then so does the third.
\end{Prop}

\begin{proof}
The standard description of additive subcategories of the additive quotient $\SKE=\cat{K}/\proj(\scE)$ tells us that $\cat{J}\mapsto \pi\inv(\cat{J})$ identifies the lattice of thick subcategories of~$\SKE$ with a sublattice of additive subcategories of~$\cat{K}$ containing~$\proj(\scE)$. We need to identify the image of this bijection, \ie express what it means for~$\cat{J}$ to be thick in~$\SKE$ in terms of~$\cat{C}=\pi\inv(\cat{J})$ in~$\cat{K}$. Being thick certainly entails being closed under direct summands. (Note that $\cat{C}$ is then automatically stable under `$\scE$-stable isomorphisms' by \Cref{lem:iso}.) For $\cat{J}$ triangulated, we need to check the 2-out-of-3 property for triangles in~$\SKE$, which amounts to the $\scE$-two-out-of-three property of the statement by the fact that every distinguished triangle in~$\SKE$ is the image-triangle of an $\scE$-triangle by \Cref{thm:Beligiannis}.
\end{proof}

\begin{Rem}
In particular, the subcategory $\cat{C}=\pi\inv(\cat{J})$ in \Cref{prop:thick-rel-stable} is not necessarily stable under suspension $\Sigma\colon \cat{K}\to \cat{K}$. (It is `stable under $\Sigma_{\scE}$' so to speak.)
\end{Rem}

\begin{Thm}
\label{thm:3rdIT-birationally}%
Under \Cref{hyp:se:thick}, let $\cat{C}$ be an additive subcategory of~$\cat{K}$, containing~$\proj(\scE)$, closed under direct summands and satisfying the \emph{$\scE$-two-out-of-three property} of \Cref{prop:thick-rel-stable} (\ie its image $\pi(\cat{C})$ is a thick subcategory of~$\SKE$). Then the following are equivalent\,:
\begin{enumerate}[\rm(i)]
\item
$\cat{C}$ is itself a thick subcategory (\ie it is a triangulated subcategory) of~$\cat{K}$.
\item
$\cat{C}$ contains the thick subcategory $\thick(\proj(\scE))$.
\end{enumerate}
In that case, that is, for every thick \emph{triangulated} subcategory~$\cat{C}\subseteq\cat{K}$ containing~$\proj(\scE)$, we have a natural equivalence induced by~$\pi$
\[
\cat{K}\Verd\cat{C}\isotoo \SKE\Verd\pi(\cat{C})
\]
between the respective Verdier quotients, which is moreover exact.
\end{Thm}

\begin{proof}
Clearly (i)$\then$(ii) since $\cat{C}$ contains~$\proj(\scE)$. Conversely, suppose that $\cat{C}$ contains~$\thick(\proj(\scE))$ and let us show that $\cat{C}$ is a triangulated subcategory of~$\cat{K}$.

Let us start with a general construction. Let $X\to Y\to Z\to \Sigma X$ be a distinguished triangle in~$\cat{K}$. Choose an $\inj(\scE)$-preenvelope $X\to I$ and consider the homotopy pushout (marked~$\square$) and the isomorphic cofibres~$Z$ in~$\cat{K}$:
\begin{equation}
\label{eq:aux1}%
\vcenter{
\xymatrix@C=2em@R=1.5em{
X \ar[r] \ar[d] \ar@{}[rd]|-{\square}
& Y \ar[d] \ar[r]
& Z \ar@{=}[d]
\\
I \ar[r]
& V \ar[r]
& Z\,.\!\!
}}
\end{equation}
Let now $V\to J$ be another $\inj(\scE)$-preenvelope and consider the `next' homotopy pushout (marked~$\square$) and the isomorphic fibres~$I$ in~$\cat{K}$:
\begin{equation}
\label{eq:aux2}%
\vcenter{
\xymatrix@C=2em@R=1.5em{
I \ar[r] \ar@{=}[d]
& V \ar[d] \ar[r] \ar@{}[rd]|-{\square}
& Z \ar[d]
\\
I \ar[r]
& J \ar[r]
& W.\!
}}
\end{equation}
The bottom triangle tells us that $W\in\thick(\inj(\scE))$ and our assumption about $\cat{C}\supseteq\thick(\proj(\scE))$ implies that~$W$ belongs to~$\cat{C}$. Now the second homotopy pushout involves an $\scE$-triangle (since $V\to J$ is an $\scE$-monomorphism)
\begin{equation}
\label{eq:aux3}%
V \to J \oplus Z \to W \to \Sigma V.
\end{equation}
This finishes the general construction.

Suppose now that $X$ and~$Y$ are in~$\cat{C}$ and let us prove that $Z\in\cat{C}$. By the $\scE$-two-out-of-three property applied to~\eqref{eq:aux1}, the object~$V$ (which is the cone of $\pi(X\to Y)$ in~$\SKE$) belongs to~$\cat{C}$. It now follows from the $\scE$-two-out-of-three property applied to~\eqref{eq:aux3}, that $J\oplus Z$, and therefore~$Z$, belong to~$\cat{C}$ as claimed. This proves (ii)$\then$(i).

Consider now the composite of the additive quotient and the Verdier quotient
\[
Q\colon \cat{K}\otoo{\pi}\SKE\otoo{q}\SKE\Verd\pi(\cat{C})\,.
\]
We claim that this functor is exact. So let $X\to Y\to Z\to \Sigma X$ be a distinguished triangle in~$\cat{K}$ and produce the diagrams~\eqref{eq:aux1}, \eqref{eq:aux2} and~\eqref{eq:aux3} as above. Since~\eqref{eq:aux3} is an $\scE$-triangle, we know that it induces an exact triangle $V\to J\oplus Z \to W \to \Sigma_\scE V$ in~$\SKE$. Since $J,W\in\cat{C}$, the image of this triangle under the exact~$q\colon\SKE\to \SKE\Verd\pi\cat{C}$ reduces to an isomorphism~$V\isoto Z$ in~$\SKE\Verd\pi\cat{C}$. More precisely, the map $V\to Z$ in~\eqref{eq:aux1} becomes an isomorphism in~$\SKE\Verd\pi(\cat{C})$. But~\eqref{eq:aux1} provides us an exact triangle~$X\to Y\to V \to \Sigma_\scE X$ in~$\SKE$ and therefore in~$\SKE\Verd\pi(\cat{C})$, in which we can now replace~$V$ by~$Z$, and use the natural transformation $\underline{\alpha}$ of Construction~\ref{cons:sigma-comparison} which is an isomorphism in~$\SKE\Verd\pi(\cat{C})$, to obtain exactly the image under~$Q$ of our original triangle (cf.\ (\ref{eq:image-triangle}) and the discussion that follows). In other words, $Q$ is exact. Along the way (or in the special case of $Y=0$) we have shown that $Q$ commutes with suspension, via the natural transformation $\underline{\alpha}$ induces.

It is easy to see that the object-kernel~$\Ker(Q)$ of the above functor $Q$ is exactly~$\cat{C}$, via \Cref{lem:iso} and $\cat{C}\supseteq\proj(\scE)$.

To summarize, the composite functor $Q\colon \cat{K}\to \SKE\Verd\pi(\cat{C})$ is an epimorphism (of additive categories) and it is exact with kernel~$\cat{C}$. Now any exact functor $F\colon \cat{K}\to \cat{L}$ which maps~$\cat{C}$ to zero, maps~$\proj(\scE)$ to zero and therefore factors via~$\pi$\,:
\[
\xymatrix@C=4em{
\cat{K} \ar[r]^-{\pi} \ar[d]_-{F}
& \SKE \ar[r]^-{q} \ar@{-->}[ld]_-{(1)}^(.4){\exists\,F'}
& \SKE\Verd\pi(\cat{C}) \ar@{-->}@/^1em/[lld]_(.4){(2)}^(.4){\exists\,F''}
\\
\cat{L}
}
\]
We claim that $F'$ is necessarily exact. Distinguished triangles in~$\SKE$ come, in particular, from \emph{some} distinguished triangles in~$\cat{K}$ up to altering the suspension as in Example~\ref{ex:st-triangle-E}. Since $F$ is exact and kills $\proj(\scE)$ it identifies $\Sigma$ and $\Sigma_\scE$ (the latter being well defined after applying $F$), and so is insensitive to our sleight of hand with the suspension. Thus $F'$ is exact as claimed. We also know that $F'(\pi\cat{C})=0$. So $F'$ factors further via~$\SKE\Verd\pi(\cat{C})$. In short, the functor $Q=q\circ\pi$ is an exact epimorphism, with kernel~$\cat{C}$ and every exact functor out of~$\cat{K}$ which kills~$\cat{C}$ factors (necessarily uniquely) via~$Q$. This uniquely characterizes the Verdier quotient $\cat{K}\Verd\cat{C}$.
\end{proof}

The following corollary recovers our original inspiration~\cite[Theorem~3.4]{BalandStevenson19}. Compare also~\cite[Proposition~6.3]{Carlson18} in modular representation theory.

\begin{Cor}\label{cor:generalBS}
Let $\cat{K}$ be a triangulated category, $\scE$ a Frobenius class of triangles, and $\pi\colon\cat{K}\to \SKE$ be the corresponding quotient functor. Let
\begin{displaymath}
\cat M = \thick_{\cat{K}}(\proj(\scE))
\end{displaymath}
be the thick subcategory of $\cat{K}$ generated by the $\scE$-projective-injectives. Then:
\begin{enumerate}[\rm(a)]
\item The full subcategory $\pi(\cat{M})$ is thick in~$\SKE$.
\item The functor $\pi$ induces an equivalence of the Verdier quotients
\begin{displaymath}
\bpi\colon \cat{K}\Verd\cat{M} \stackrel{\sim}{\to} \SKE\Verd\pi \cat{M}\,.
\end{displaymath}
\item The functor $\bpi$ is exact.
\qed
\end{enumerate}
\end{Cor}

\begin{Rem}
\label{rem:birational}
Let us emphasize the content of \Cref{thm:3rdIT-birationally} and \Cref{cor:generalBS}. We have seen, in \Cref{lem:notexact1}, that $\pi\colon \cat{K}\to \SKE$ is not exact outside of degenerate situations. Yet, \Cref{thm:3rdIT-birationally} shows that one can compare some Verdier quotients of~$\cat{K}$ with the corresponding Verdier quotients of~$\SKE$, via exact functors. The class of permissible quotients is defined by the thick subcategories containing a minimal one: $\cat{M}=\thick(\proj(\scE))$. In fact, if there exists a functor~$F\colon \SKE\to \cat{L}$ such that the composite $F\circ\pi\colon \cat{K}\to \cat{L}$ is exact then $F\circ \pi$ vanishes on~$\cat{M}$ (since it is exact and vanishes on $\proj(\scE)$) and therefore induces an exact functor $\cat{K}\Verd\cat{M}\to \cat{L}$. This shows that $\cat{K}\Verd\cat{M}$ is the largest quotient of~$\cat{K}$ `on which $\pi$ (or anything factoring via~$\pi$) is exact'. \Cref{cor:generalBS} tells us that this Verdier quotient of~$\cat{K}$ is indeed also a Verdier quotient of~$\SKE$, with respect to the corresponding thick subcategory~$\pi\cat{M}$.

The existence of the exact equivalence $\xymatrix@C=2em{\cat{K}\ar@{-->}[r]&\SKE}$ defined only on suitable quotients modulo thick subcategories reminds us of birationality in algebraic geometry, \ie of a morphism which is only defined on certain open subschemes, `away' from $\cat{M}=\thick(\proj(\scE))$. We shall return to this `birationality' result in \Cref{se:birationality}, when we have the conceptual force of tt-geometry behind us.
\end{Rem}

\begin{Rem}\label{rem:noBS}
It can happen that $\cat M = \cat{K}$ in \Cref{cor:generalBS}, \ie that $\proj(\scE)$ generates~$\cat{K}$, in which case, \Cref{thm:3rdIT-birationally} is largely devoid of content.
\end{Rem}

\threestars

Let us continue on in the setting of Hypothesis~\ref{hyp:se:thick}. It is natural to wonder what we can say about the lattice of thick subcategories of $\SKE$ in terms of the lattice for~$\cat{K}$. We learned in Proposition~\ref{prop:thick-rel-stable} that there is a bijective correspondence between the thick subcategories of $\SKE$ and certain additive subcategories of $\cat{K}$, namely those containing~$\proj(\scE)=\inj(\scE)$, closed under direct summands and satisfying the $\scE$-two-out-of-three property. But, these are not necessarily triangulated; the ones which are triangulated are precisely those containing $\thick(\proj(\scE))$ by Theorem~\ref{thm:3rdIT-birationally}. It turns out that there can be many of the former, non-$\cat{K}$-triangulated, variety and that the lattices can exhibit starkly divergent behaviors. The following construction hints at this and was discovered by Jon Carlson in representation theory, see~\cite{Carlson18}.

\begin{Prop}\label{prop:Jon}
Let $G\colon\cat{K} \to \cat{N}$ be an exact functor between triangulated categories such that every triangle in $\scE$ is $G$-split, i.e. the $\scE$-phantoms are contained in~$\ker G$. Let $\cat A \subseteq \cat{N}$ be an additive subcategory of $\cat{N}$ which contains $G(\proj(\scE))$ and is closed under direct summands. Then 
\begin{displaymath}
\pi(G^{-1}(\cat A)) = \{X\in \SKE \;\vert\; X\cong \pi Y \text{ with } Y\in \cat{K} \text{ and } GY\in \cat A\},
\end{displaymath}
is thick in~$\SKE$.
\end{Prop}

\begin{proof}
By assumption the additive subcategory~$\cat{C}:=G\inv(\cat{A})$ of~$\cat{K}$ contains~$\proj(\scE)$ and is closed under direct summands. We claim that $\cat{C}$ satisfies the $\scE$-two-out-of-three property of \Cref{prop:thick-rel-stable}. This is easy since any $\scE$-triangle becomes split under~$G$ and $\cat{A}$ is closed under sums and summands. Then \Cref{prop:thick-rel-stable} tells us that $\pi\cat{C}$ is then a thick subcategory of~$\SKE$.
\end{proof}

\begin{Rem}
A simple example, which is already of interest, is given by supposing we are in the Wirthm\"ullery situation of \Cref{thm:Wirthmuller} and taking $G=F$.
\end{Rem}

\begin{Rem}\label{rem:Jon}
Of course, the above construction $\Phi\colon \cat A \mapsto \pi(G^{-1}\cat A)$ preserves inclusions. Thus, we can start with a proper chain of suitable categories
\begin{displaymath}
\cat A_1 \subsetneq \cdots \subsetneq \cat A_n \subsetneq \cat A_{n+1} \subsetneq \cdots
\end{displaymath}
in $\cat{N}$ and obtain a chain of \emph{thick} subcategories of $\SKE$
\begin{displaymath}
\Phi\cat A_1 \subseteq \cdots \subseteq \Phi\cat A_n \subseteq \Phi\cat A_{n+1} \subseteq \cdots
\end{displaymath}
The rub is that our chain may no longer be proper, which naturally leads to the following question. Under which conditions can we say that $\cat A_1 \subsetneq \cat A_2$ implies $\Phi\cat A_1 \subsetneq \Phi\cat A_2$? For an example where this does happen, we refer to \cite[proof of Theorem~5.1]{Carlson18}.
\end{Rem}

\goodbreak
\section{Algebraicity}
\label{se:algebraicity}%
\medbreak

In this section we will prove \Cref{thm:algebraicity-intro}, showing that passing to the relative stable category preserves the property of being algebraic. For a refresher on exact categories see \Cref{rec:exact}.

\begin{Thm}
\label{thm:algebraicity}%
Let $\cat{F}$ be a Frobenius exact category, with admissible exact sequences~$\scS$, let $\cat{K}=\underline{\cat{F}}=\cat{F}/\proj(\scS)$ be its stable category and let $\varpi\colon \cat{F}\to \cat{K}$ be the quotient functor. Suppose that $\scE$ is a Frobenius class of triangles in~$\cat{K}$. Define~$\SE\subseteq\scS$ to be the following class of admissible sequences in~$\cat{F}$:
\[
\SE:=\SET{\EE=\big(M'\overset{i}\into M\overset{p}\onto M''\big)\in\scS}{\textrm{the image triangle }\varpi\EE\textrm{ belongs to }\scE\textrm{ in }\cat{K}}\,.
\]
Then $\SE$ defines a Frobenius exact structure on~$\cat{F}$ whose projective-injective objects $\varpi\inv(\proj(\scE))$ are those objects of~$\cat{F}$ which are $\scE$-projective-injective in~$\cat{K}$. Moreover, the stable category of this new Frobenius category is equivalent to~$\SKE$ as a triangulated category in such a way that the canonical diagram commutes\,:
\[
\xymatrix@R=1.5em{
\cat{F} \ar[r]^-{\varpi} \ar[d]_-{\varpi'}
& \cat{F}/\proj(\scS)=\cat{K} \ar[d]^-{\pi}
\\
\cat{F}/\proj(\SE) \ar[r]^-{\cong}
& \SKE\,.\!\!
}
\]
\end{Thm}

\begin{proof}
It is easy, as we outline below, to check that $\SE$ is indeed an exact structure: Some of the axioms are inherited from~$\scS$ and the rest are verified by applying~$\varpi\colon \cat{F}\to \cat{K}$ and using the axioms of a proper class of triangles for~$\scE$. By definition (admissible) $\SE$-monomorphisms are precisely the (admissible) $\scS$-monomorphisms whose image under~$\varpi$ are $\scE$-monomorphisms. And similarly with epimorphisms.

So, it is straightforward that $\SE$ is closed under isomorphisms and contains the split exact sequences by additivity of~$\varpi$ and since~$\scE$ has these properties. We also inherit from $\scS$ that sequences in $\SE$ are kernel-cokernel pairs. Similarly, $\SE$-monomorphisms and $\SE$-epimorphisms are closed under composition, since this holds for~$\scS$, and is true of $\scE$-monomorphisms and $\scE$-epimorphisms. Finally, let us prove that pushouts of $\SE$-monomorphisms exist and are $\SE$-monomorphisms, the pullbacks of $\SE$-epimorphisms being dual. Again, existence is inherited from~$\scS$ and provides $\scS$-monomorphisms that still need to be shown to be $\SE$-admissible. We want to test this under~$\varpi$ by using the analogous property for~$\scE$-monomorphisms. To this end, it suffices to know that a pushout square along an $\scS$-monomorphism yields a homotopy pushout square in~$\cat{K}$, which is immediate from the formulation of those squares in terms of $\scS$-exact sequences and distinguished triangles respectively, and the fact that $\varpi$ maps $\scS$-exact sequences to distinguished triangles.

Let us now see why $\SE$-projectives coincide with $\varpi\inv(\proj(\scE))$. Let $Q\in\cat{F}$ be an object and consider an $\scS$-exact sequence~$\EE=\big(M'\overset{i}\into M\overset{p}\onto M''\big)$ and its image distinguished triangle in~$\cat{K}$, for some morphism~$h$
\[
\varpi\EE=\big(\xymatrix{\varpi M' \ar[r]^-{\varpi (i)} &\varpi M \ar[r]^-{\varpi (p)} &\varpi M'' \ar[r]^-{h} & \Sigma\varpi M'
}\big).
\]
The connecting map~$\partial$ in the $\Ext^\sbull_{\cat{F}}(Q,-)$ long exact sequence induced by~$\EE$
\[
\partial\colon\cat F(Q, M'') \to \Ext^1_{\cat F}(Q, M')
\]
factors as the following composite
\[
\xymatrix{
\cat F(Q,M'') \ar[r]^-{\varpi} & \cat{K}(\varpi Q, \varpi M'') \ar[rr]^-{\cat{K}(\varpi Q, h)} && \cat{K}(\varpi Q, \Sigma \varpi M')\cong \Ext^1_{\cat F}(Q,M').
}
\]
As the first map is surjective, we see that the following two properties are equivalent:
\begin{enumerate}[(A)]
\item
The homomorphism $\cat{K}(\varpi Q,h)$ is zero.
\item
The connecting map $\partial$ is zero, \ie the following short sequence is exact:
\[
0\to \cat{F}(Q,M')\to \cat{F}(Q,M)\to \cat{F}(Q,M'')\to 0\,.
\]
\end{enumerate}
We now vary the quantifiers: We fix~$Q$ but we let $\EE$ vary among all $\SE$-exact sequences, which is equivalent to $\varpi\EE$ varying among all $\scE$-triangles since any triangle in the stable category is isomorphic to the image of some $\scS$-exact sequence. Property (A) for all such~$\EE$ (that is, for all $\scE$-phantoms $h$) says that $\varpi Q$ is $\scE$-projective; see \Cref{prop:basic-proper}. Property~(B) for all such~$\EE$ says that $Q$ is $\SE$-projective. This proves that the class of $\SE$-projectives coincides with $\varpi\inv(\proj(\scE))$.

A dual argument proves that the class of $\SE$-injectives coincides with $\varpi\inv(\inj(\scE))$. Since $\scE$ is Frobenius, it also follows that $\SE$-projectives and $\SE$-injectives coincide. We now need to show there are enough of them.

Let $X$ be an object of $\cat F$. We can find a $\proj(\scE)$-precover of~$\varpi X$ in~$\cat{K}$, which can be written $\varpi(q)\colon \varpi Q\to \varpi X$ for some $Q\in \varpi\inv(\proj(\scE))$ and some morphism $q\colon Q\to X$ in~$\cat{F}$. By the above discussion, we already know that $Q$ is $\SE$-projective. Replacing $q\colon Q\to X$ by $\smat{q&p}\colon Q\oplus P\to X$ where~$p\colon P\to X$ is an $\scS$-projective cover, we can assume that $q\colon Q\to X$ is an $\scS$-epimorphism in~$\cat{F}$, without changing the image~$\varpi(q)$ since $\varpi(P)=0$. Consider the associated $\scS$-exact sequence in~$\cat{F}$
\[
\EE=\big(Y \overset{j}\into Q \overset{q}\onto X\big)\,.
\]
Its image under~$\varpi$ is a distinguished triangle
\[
\varpi\EE=\big(\xymatrix{
\varpi Y \ar[r]^-{\varpi(j)}
&\varpi Q \ar[r]^-{\varpi(q)}
&\varpi X \ar[r]^-{k}
& \Sigma \varpi Y
}\big)
\]
for some morphism~$k$. Since $\varpi(q)$ was chosen to be a $\proj(\scE)$-precover, it is an $\scE$-epimorphism and this triangle belongs to~$\scE$. By definition of~$\SE$, the sequence~$\EE$ then belongs to~$\SE$. In short, $q\colon Q\to X$ is an $\SE$-precover. Dually, $\cat{F}$ admits enough $\SE$-injectives. So $\cat{F}$ is also Frobenius with the (smaller) class of exact sequences~$\SE$.

It only remains to verify that the $\SE$-stable category of $\cat{F}$ is equivalent to~$\SKE$ as triangulated categories. But this is clear from the fact that the $\SE$-projective-injective objects are precisely those in~$\mcQ:=\varpi\inv(\proj(\scE))$, and a standard `third isomorphism theorem' for additive quotients\,:
\[
\cat{K}/\proj(\scE)=(\cat{F}/\proj(\scS))\big/(\mcQ/\proj(\scS))\cong \cat{F}/\mcQ.
\]
Exactness of this equivalence is left as an easy exercise.
\end{proof}

\goodbreak
\section{Tensor-triangulated stable categories}
\label{se:tt-stab}%
\medbreak

From now on, we assume that $\cat{K}$ is not only triangulated but \emph{tensor}-triangulated. We discuss what to assume on~$\cat{K}$ and~$\scE$ to make sure that the relative stable category~$\SKE$ becomes tensor-triangulated. We also discuss how the tensor on~$\cat{K}$ naturally produces some interesting Frobenius classes~$\scE$. We then discuss the impact on the triangular spectra~$\SpcK$ and $\Spc(\SKE)$.

Our assumption that $\cat{K}$ is a \emph{tensor-triangulated} category (\emph{tt-category} for short) means that $\cat{K}$ admits a symmetric monoidal structure
\[
\otimes\colon \cat{K}\times \cat{K}\too \cat{K}
\]
which is exact in each variable, \ie $X\otimes ?\colon \cat{K}\to \cat{K}$ is exact for every $X\in\cat{K}$, and such that the two induced isomorphisms $\Sigma ? \otimes \Sigma ?? \isoto \Sigma^2(?\otimes ??)$ commute up to sign. See details in~\cite[App.\,A]{HoveyPalmieriStrickland97}. We denote the $\otimes$-unit by~$\unit$.

\begin{Def}
\label{def:rigid}%
An object $X\in\cat{K}$ is called \emph{rigid} (\aka \emph{strongly dualizable}) if there exists a object $X^\vee\in\cat{K}$, called its \emph{dual}, and an adjunction
\begin{equation}
\label{eq:adj-X}%
\vcenter{\xymatrix@R=1.5em{
\cat{K} \ar@<-.3em>[d]_-{X\otimes ?}
\\
\cat{K} \ar@<-.3em>[u]_-{X^\vee\otimes ?}
}}
\end{equation}
Of interest to us will be the unit and counit of this adjunction (at~$\unit$), denoted
\[
\coev\colon \unit\to X^\vee\otimes X
\qquadtext{and}
\ev\colon X\otimes X^\vee \to \unit\,.
\]
We say that $\cat{K}$ is \emph{rigid} if every object~$X\in\cat{K}$ is rigid and the functor $(-)^\vee$ is exact; rigidity implies $\cat{K}$ is closed monoidal and exactness of~$(-)^\vee$ is equivalent to the internal hom being exact in both variables. Standard monoidal adjunction theory gives canonical isomorphisms $(X\otimes Y)^\vee\cong X^\vee\otimes Y^\vee$ and $X^{\vee\vee}\cong X$, etc.
\end{Def}

\begin{Rem}
\label{rem:monoidal-functors-preserve-rigidity}%
It is well-known that if $F\colon \cat{K}\to \cat{L}$ is monoidal and $X$ is rigid in~$\cat{K}$ then $F(X)$ is rigid in~$\cat{L}$, with dual $F(X)^\vee=F(X^\vee)$. Indeed, it suffices to transport the unit and counit of the adjunction (and their relations) under~$F$.
\end{Rem}

The compatibility of proper classes of triangles with tensor is the obvious one:
\begin{Def}
\label{def:tt-proper}%
We say that a proper class of triangles~$\scE$ as in \Cref{hyp:E} is a \emph{\ttclass} if it is closed under tensoring in the sense that if
\begin{displaymath}
\xymatrix{
X \ar[r]^-f & Y \ar[r]^-g & Z\ar[r]^-h & \Sigma X
}
\end{displaymath}
belongs to~$\scE$ and $W\in \cat{K}$ then the triangle
\begin{displaymath}
\xymatrix{
X \otimes W \ar[r]^-{f\otimes W} & Y\otimes W \ar[r]^-{g\otimes W} & Z\otimes W \ar[r]^-{h\otimes W} & \Sigma X \otimes W
}
\end{displaymath}
also lies in~$\scE$. This is equivalent to asking that the class of $\scE$-monomorphisms (or the class of $\scE$-epimorphisms or the ideal of phantoms $\Ph_{\scE}$) is closed under tensoring. Since $\Ph_{\scE}$ is an ideal, it follows from being closed under tensoring with objects (\ie identity morphisms) that it is $\otimes$-closed, \ie closed under tensoring with all maps.
\end{Def}

\begin{Rem}
It is however not immediate from $\scE$ being a Frobenius \ttclass\ that $\SKE$ becomes a tt-category. Indeed, we still need to show that $\proj(\scE)$ is tensor-closed.
\end{Rem}

\begin{Def}
\label{def:hom-closed}%
If $\cat{K}$ has an internal hom, which we denote $\hom\colon \cat{K}\op\times\cat{K}\too\cat{K}$, it also makes sense to ask for $\scE$ to be closed under $\hom(W,-)$ for all $W\in \cat{K}$. In this case we say $\scE$ is \emph{$\hom$-closed} and evidently this corresponds to $\hom$-closure of $\Ph_{\scE}$.
\end{Def}

\begin{Lem}\label{lem:ttproj}
If $\scE$ is a proper $\hom$-closed class of triangles then $\proj(\scE)$ is closed under tensoring. In particular, the ideal $\mcN$ of maps factoring via $\proj(\scE)$ is a $\otimes$-closed. Dually, if $\scE$ is a proper \ttclass\ of triangles, then $\inj(\scE)$ is $\hom$-closed. In particular, the ideal $\mcN'$ of maps factoring via $\inj(\scE)$ is $\hom$-closed.

If moreover $\scE$ has enough projectives (injectives) then the above implications are equivalences.
\end{Lem}

\begin{proof}
This is an immediate consequence of the tensor-hom adjunction.
\end{proof}

In the situation of \Cref{lem:ttproj}, provided $\scE$ is Frobenius, the stable category $\SKE$ is again symmetric monoidal as is the quotient functor $\pi$; this is a general fact about additive quotients by tensor-closed ideals of morphisms. Moreover, rigid objects of $\cat{K}$ get sent to rigid objects in $\SKE$; this is a consequence of the fact that $\pi$ is monoidal (\Cref{rem:monoidal-functors-preserve-rigidity}).

\begin{Prop}\label{prop:ttquotient}
If $\scE$ is a Frobenius class of triangles closed under $\hom$ and tensoring then the category $\SKE$, equipped with the induced triangulation, is closed symmetric monoidal in such a way that both hom and tensor are exact in each variable. If, moreover, $\cat{K}$ is rigid then so is~$\SKE$.
\end{Prop}
\begin{proof}
It is, as mentioned above, well known that $\SKE$ is again symmetric monoidal and the quotient functor is symmetric monoidal. By \Cref{lem:ttproj} we know $\proj(\scE)$ is closed under tensoring and, by assumption, so is~$\scE$. Thus \Cref{prop:naturality} tells us that the induced tensor product on $\SKE$ is exact in both variables.

As we have noted, by \Cref{lem:ttproj}, $\proj(\scE)$ is closed under tensoring. Since $\scE$ is Frobenius, said lemma also tells us that $\proj(\scE)$ is closed under the internal hom. We can thus apply \Cref{lem:adjoints-descend} to see that the internal hom descends to an internal hom on~$\SKE$. It is immediate that $\pi$ is closed in this situation.

Finally, the statement concerning rigidity follows from \Cref{rem:monoidal-functors-preserve-rigidity}.
\end{proof}

\begin{Rem}
Note that we do not assert that $\SKE$ is quite a tt-category: We do not know that, in $\SKE$, the tensor product satisfies the appropriate Koszul signs with respect to suspension. In certain general situations we can verify this, see Corollary~\ref{cor:Ksigns}.
\end{Rem}

\begin{Cor}\label{cor:ttproj}
If $\cat{K}$ is rigid and $\scE$ is a proper \ttclass\ of triangles, the subcategories of projectives and of injectives, $\proj(\scE)$ and $\inj(\scE)$, are tensor-closed. In that case, $\SKE$ is a rigid tt-category (up to Koszul signs).
\qed
\end{Cor}

\threestars

Let us now turn to the other connection between the tensor structure of~$\cat{K}$ and proper classes: The production of Frobenius classes of triangles using the tensor. In glorious generality, one can finesse a tensor-version of the Wirthm\"uller conditions of \Cref{thm:Wirthmuller}. Let us instead go straight for the jugular.

\begin{Cor}\label{cor:rigid}
Let $\cat{K}$ be a tt-category and let $B\in \cat{K}$ be a \emph{rigid} object. Let $\scE_B$ denote the class of $(B\otimes-)$-split triangles, that we shall simply call \emph{$B$-split triangles}, \ie the proper class whose ideal of phantoms $\Ph_B=\ker(B\otimes-)$ are the maps killed by tensoring with~$B$. Then $\scE_B$ is a Frobenius \ttclass\ of triangles, whose class of projective-injectives is precisely $\add(B\otimes\cat{K})$. The corresponding quotient
\[
\SKB=\cat{K}/\add(B\otimes\cat{K})
\]
is a tt-category, and the quotient functor $\pi_B\colon \cat{K}\to \SKB$ is monoidal.
\end{Cor}
\begin{proof}
Since $B$ is rigid we have adjoint functors
\begin{displaymath}
\xymatrix{
\cat{K}
 \ar[d]|-{B\otimes -}
\\
\cat{L}
 \ar@<-1.5em>[u]_-{B^\vee\otimes -}
 \ar@<1.5em>[u]^-{B^\vee\otimes -}
}
\end{displaymath}
and so \Cref{thm:Wirthmuller} applies to give the first statement. We use here that the unit-counit relations for the adjunction tell us that $\add(B\otimes\cat{K}) = \add(B^\vee\otimes\cat{K})$ (cf.\ \Cref{prop:rigidfacts}). For the statements regarding the monoidal structure it is enough to note that if $f\in\ker(B\otimes-)$ is a $B$-phantom and $X\in \cat{K}$ is arbitrary then
\begin{displaymath}
B\otimes (X\otimes f) \cong (B\otimes f)\otimes X = 0
\end{displaymath}
and
\begin{displaymath}
B \otimes \hom(X,f) \cong \hom(B^\vee, \hom(X,f)) \cong \hom(X\otimes B^\vee, f) \cong \hom(X, B\otimes f) = 0.
\end{displaymath}
Thus \Cref{prop:ttquotient} applies to give the remaining statements with the exception of the Koszul signs. This additional compatibility condition is satisfied, but we defer the proof to Corollary~\ref{cor:Ksigns}.
\end{proof}

\begin{Exa}
\label{exa:representations}%
Let~$p$ be a prime number and $k$ be a field of characteristic~$p$. Let $G$ be a finite group (we assume that $p$ divides the order of $G$ to avoid dwelling too much on the zero category). Given a finite collection $\mcH$ of subgroups of $G$ one can consider $\smodu_\mcH kG$, the stable category of $G$ relative to $\mcH$ as defined in~\cite{CarlsonPengWheeler98}. This is given by working relative to modules induced from subgroups in $\mcH$ or, equivalently, by considering $\modu kG$ with the exact structure where the exact sequences are those that split when restricted to all subgroups in~$\mcH$.

This construction can be phrased in terms of the additive quotients studied here. The sum, suspension, and summand closed subcategory of $\smodu kG$
\begin{displaymath}
\mcC_\mcH = \add(\Ind_H^G M\;\vert\; H\in \mcH, M\in \smodu kH)
\end{displaymath}
is preenveloping and precovering (this is proved in~\cite{CarlsonPengWheeler98}), so the corresponding additive quotient is triangulated. One easily checks that it coincides with $\smodu_\mcH kG$.

More generally, one could take a finite dimensional representation $V$ of $G$ and consider $((\smodu kG)\otimes V)^\natural$, which is again the projective-injectives for a Frobenius class of triangles and obtain a corresponding triangulated additive quotient. The interested reader is referred for instance to~\cite{Lassueur11}.
\end{Exa}

Let us return to our general discussion. We begin with some basic observations.

\begin{Prop}\label{prop:rigidfacts}
Let $B\in \cat{K}$ be rigid.
\begin{enumerate}[\rm(a)]
\item The rigid object $B$ is $\otimes$-faithful (\ie $B\otimes -\colon \cat{K} \to \cat{K}$ is faithful) if and only if $\coev\colon \unit \to B \otimes B^\vee$ is a split monomorphism, if and only if $\SKB = 0$.
\smallbreak
\item The objects $B$ and $B^\vee$ and, more generally, the rigid objects $B^{\otimes m}\otimes (B^\vee)^{\otimes n}$ for all~$m,n >0$ define the same relative category $\SKB$.
\end{enumerate}
\end{Prop}
\begin{proof}
Part~(a) is an easy exercise using that the only monomorphisms in a triangulated category are the split ones. Part~(b) follows from the standard unit-counit relations that guarantee $B\in \add(B^\vee\otimes\cat{K})$, etc.
\end{proof}

\begin{Rem}\label{rem:noBS2}
If $B$ is only \emph{nil-faithful} (\aka solid), meaning that every $f\in\ker(B)$ is $\otimes$-nilpotent, which is also equivalent to saying $\supp B = \Spc \cat{K}$, then $\SKB$ is not necessarily trivial. In other words, even when $\thick^\otimes(B)=\cat{K}$, if $\add(B\otimes\cat{K})\neq \cat{K}$ then $\SKB\neq 0$ is a non-trivial tt-category.
\end{Rem}

\begin{Exa}
In the representation-theoretic setting of \Cref{exa:representations}, consider a collection $H_1,\ldots, H_n\le G$ of subgroups and the object $B=\oplus_{i=1}^n\Ind_{H_i}^G k$ in $\cat{K}=\smodu(kG)$. Then $B$ is $\otimes$-faithful if only if one of the indices $[G:H_i]$ is prime to~$p$, whereas the property $\supp(B)=\SpcK$ is much more flexible: It suffices that every elementary abelian subgroup of~$G$ is conjugate to a subgroup of one the~$H_i$.
\end{Exa}

We can specialize much of what we have seen in the general triangular setting to the tt-setting. For instance, we have the obvious analogue of naturality:

\begin{Cor}\label{cor:ttnaturality}
Given a tt-functor $F\colon \cat{K} \to \cat{K}'$ and $B\in \cat{K}$ rigid then there is a unique tt-functor $\underline{F}$ making the following square commute
\begin{displaymath}
\xymatrix{
\cat{K}\ar[r]^-F \ar[d]_-\pi & \cat{K}' \ar[d]^-{\pi'} \\
\SKB \ar[r]_-{\underline{F}} & \STAB{\cat{K}'}{F(B)}\,.\kern-1em
}
\end{displaymath}
\end{Cor}

\begin{proof}
It is enough to check conditions~\eqref{it:naturality-1} and~\eqref{it:naturality-2} of \Cref{prop:naturality} as the preservation of the monoidal structure follows from the universal property of the monoidal additive quotients. Let $B'=F(B)$ and let $\scE_B$ and $\scE_{B'}$ denote the Frobenius classes of triangles arising from $B$ and $B'$ in~$\cat{K}$ and $\cat{K}'$ respectively. If $\delta$ is a $B$-phantom, \ie $B\otimes \delta=0$, then since $F$ is monoidal we have $B'\otimes F(\delta) \cong F(B)\otimes F(\delta) \cong F(B\otimes \delta) \cong 0$. Thus $F(\scE_B)\subseteq \scE_{B'}$, showing~\eqref{it:naturality-1}, by exactness of~$F$. Similarly, if $Y$ is a summand of $X\otimes B$ then $F(Y)$ is a summand of $F(B\otimes X)= B'\otimes F(X)$ so $F$ sends $\add(B\otimes\cat{K})$ to $\add(B'\otimes\cat{K}')$ showing~\eqref{it:naturality-2} holds.
\end{proof}

\begin{Rem}
One might hope that \Cref{cor:ttnaturality} would provide a method for producing equivalences of relative stable categories, or facilitating their computation by finding a good $F\colon \cat{K}\to \cat{K}'$. However, asking that $\underline{F}$ be conservative is already a lot to hope for if $F$ is not already conservative itself. Indeed, if $\Ker \underline{F} = 0$ then $\Ker F \subseteq \add(B\otimes\cat{K})$. In particular, if $\Ker \underline{F} = 0$ but $\Ker F\neq 0$ then $\add(B\otimes\cat{K})$ contains a non-zero tt-ideal. This property reminds us of \Cref{thm:3rdIT-birationally} and offers a natural transition to the next section.
\end{Rem}

\goodbreak
\section{Thick tensor-ideals and birationality}
\label{se:birationality}%
\medbreak

We now concentrate on what we can say concerning thick $\otimes$-ideals in the case that the relative quotient is a tt-category with the induced tt-structure. We focus on the setup of \Cref{cor:rigid} \emph{with the additional assumption that $\cat{K}$ is rigid}: So $\cat{K}$ is a rigid tt-category, $B\in \cat{K}$ is an object, which gives rise to the Frobenius class of $B$-split triangles $\scE_B$ with $\proj(\scE_B) = \add(B\otimes\cat{K})$, and the corresponding stable category $\SKB=\cat{K}/\add(B\otimes\cat{K})$ is again a rigid tt-category. We denote by $\pi\colon\cat{K}\to \SKB$ the (non-exact) monoidal quotient functor.

\begin{Rem}
\label{rem:K(U)}%
In this situation the thick subcategory of \Cref{cor:generalBS}
\begin{displaymath}
\cat M = \thick_{\cat{K}}(\proj(\scE_B)) = \thick_{\cat{K}}(B\otimes\cat{K}) = \thick^\otimes_{\cat{K}}(B)
\end{displaymath}
is automatically a $\otimes$-ideal. Thus the corresponding equivalence between the quotients can be idempotent completed to an equivalence
\begin{displaymath}
\cat{K}(U) = (\cat{K}/\cat{K}_Z)^\natural \cong (\SKB / \pi \cat M)^\natural,
\end{displaymath}
where $\natural$ denotes idempotent completion, and $\cat{K}(U)$ is the category associated to the open $U = \Spc \cat{K}\setminus Z$ for $Z= \supp_{\cat{K}} B$.

This helps to make precise our comments about birationality in \Cref{rem:birational}: $\SKB$ admits a Verdier quotient equivalent to the part of $\cat{K}$ supported on~$U$. We want to show that the righthand category $(\SKB / \pi \cat M)^\natural$ is also $\SKB(V)$ for a natural quasi-compact open~$V\subseteq\Spc(\SKB)$.
\end{Rem}

Now let us consider a construction which exploits the tensor structure in a way which is different to what we have seen so far and that will help us determine the open~$V$ of \Cref{rem:K(U)}.
\begin{Cons}
\label{cons:F_B}%
Ponder the distinguished triangle on the coevaluation morphism
\begin{equation}
\label{eq:F1B}%
\xymatrix{
F_B \ar[r]^-{\xi_B} & \unit \ar[r]^-{\coev} & B\otimes B^\vee \ar[r] & \Sigma F_B
}
\end{equation}
in~$\cat{K}$. By the unit-counit relations for an adjunction we know $\coev\otimes B$ is a split monomorphism, and so its fibre $\xi_B$ is $B$-phantom:
\begin{equation}
\label{eq:B@xi_B=0}%
\xi_B\otimes B=0\,.
\end{equation}
We already get a lot of mileage out of the morphism~$\xi_B$. This map controls the triangulated structure on~$\SKB$; in some sense one can view $\SKB$ as being $\cat{K}$ `twisted by~$\xi_B$'. This is made precise in the next theorem.
\end{Cons}

Recall from Example~\ref{exa:auto} that we denote by $\underline{\Sigma}$ the autoequivalence of $\SKB$ induced by $\Sigma$, the suspension of $\cat{K}$, which is naturally isomorphic to $\underline{\Sigma}\unit \otimes (-)$. We will also need to concept of $(\otimes$-)invertibility: an object $X\in \cat{K}$ is \emph{invertible} if there exists an $X^{-1}$ such that $X\otimes X^{-1} \cong \unit$. We note that if $X$ is invertible then necessarily $X^{-1} = X^\vee$.

\begin{Thm}
\label{thm:triangulation-SKB}%
With the above notation, we have:
\begin{enumerate}[\rm(a)]
\item
\label{it:FB-invertible}%
The object $\pi(F_B)$ is invertible in $\SKB$.
\smallbreak
\item
\label{it:suspension-SKB}%
There is a natural isomorphism
\begin{displaymath}
\Sigma_B(-) \cong \pi(F_B)\otimes \underline{\Sigma}(-)
\end{displaymath}
between the suspension in~$\SKB$ and the functor induced by suspension (\Cref{exa:auto}) twisted by this invertible. Moreover, under this isomorphism, $\pi(\xi_B)\otimes \underline{\Sigma}\colon \Sigma_B\to \underline{\Sigma}$ is the comparison morphism~$\alpha$ of \Cref{cons:sigma-comparison}.
\smallbreak
\item
\label{it:ucone}%
Given a morphism $f\colon X\to Y$ in~$\cat{K}$, the cone on its image $\pi f$ in~$\SKB$ can be obtained as $\pi\big(\cone(\xi_B \otimes f))$. In fact the whole triangle over~$\pi(f)$ can be obtained from the diagram
\begin{displaymath}
\xymatrix{
& X \ar[d]_-f & & \\
F_B\otimes X \ar[r]_-{\xi_B\otimes f} & Y \ar[r]^-g & Z \ar[r]^-h & \Sigma(F_B\otimes X),
}
\end{displaymath}
where the row is a distinguished triangle in $\cat{K}$, by applying $\pi$ to the triple $(f,g,h)$ and using the identification ${\pi(F_B\otimes \Sigma)}\cong \Sigma_B$ of part~\eqref{it:FB-invertible}.
\end{enumerate}
\end{Thm}

\begin{proof}
The map $\xi_B$ is a $B$-phantom and so $\coev$ is a $\inj(\scE_B)$-preenvelope. It follows that in $\SKB$ we have
\begin{displaymath}
\Sigma_B\unit \cong \pi(\Sigma F_B).
\end{displaymath}
This generalizes to any object as $\coev$ gives functorial $\inj(\scE_B)$-preenvelopes by Lemma~\ref{lem:adjunction}, proving the second part of the statement.
Let us emphasize that $\Sigma_B \unit$ is generally \emph{not} $\underline{\Sigma}\unit=\pi(\Sigma\unit)$, although both are $\otimes$-invertible in~$\SKB$: The former because it is the suspension of the unit, the latter because $\pi$ is monoidal.

To show $\pi(F_B)$ is invertible it's enough to note that desuspending the triangle on $\coev$ gives an isomorphism
\begin{displaymath}
\pi(F_B) \cong \Sigma_B\underline{\Sigma}^{-1}\unit
\end{displaymath}
and then use that $\pi$ and suspension preserve invertibility.

It remains to prove (c). By what we have already seen, we may compute the cone of $\pi(f)\colon \pi X \to \pi Y$ via the central homotopy bicartesian square in the following map of triangles
\begin{displaymath}
\xymatrix{
F_B\otimes X \ar@{=}[d] \ar[rr]^-{\xi_B\otimes X} && X \ar[d]_-f \ar[rr]^-{\coev_X} && B\otimes B^\vee \otimes X \ar[r] \ar[d] & \Sigma(F_B\otimes X) \ar@{=}[d] \\
F_B\otimes X \ar[rr]_-{\xi_B\otimes f} && Y \ar[rr]^-g && Z \ar[r]^-h & \Sigma(F_B\otimes X) \\
}
\end{displaymath}
where the fibres agree by the octahedral axiom. It follows that $\pi(Y \otoo{g} Z) = \pi\big(\cone(\xi_B \otimes f))$ is the cone of $\pi(f)$ in $\SKB$.
\end{proof}

\begin{Cor}\label{cor:Ksigns}
Let $\cat{K}$ be a rigid tt-category and $B\in \cat{K}$ an object. Then the tensor product and suspension are related by the usual Koszul signs in $\SKB$.
\end{Cor}
\begin{proof}
By Theorem~\ref{thm:triangulation-SKB}(b) there is an isomorphism $\Sigma_B(-) \cong \pi(F_B)\otimes \underline{\Sigma}(-)$. It follows immediately from the fact that $\cat{K}$ is a tt-category that $\underline{\Sigma}$ plays nicely with $\otimes$. Thus it is sufficient to verify that juggling $\pi(F_B)$, via associators and symmetry constraints which we denote by $a$ and $s$ respectively, does not introduce any signs i.e.\ that the following diagram commutes.
\begin{displaymath}
\xymatrix@R=2.1em{
(F_B^{\otimes r}\otimes X) \otimes (F_B^{\otimes s}\otimes Y)  \ar[r]^-s \ar[dd]_-a  & (F_B^{\otimes r}\otimes X) \otimes (Y\otimes F_B^{\otimes s}) \ar[d]^-a \\
 & ((F_B^{\otimes r}\otimes X)\otimes Y)\otimes F_B^{\otimes s} \ar[d]^-s \\
F_B^{\otimes r} \otimes (X\otimes(F_B^{\otimes s} \otimes Y)) \ar[d]_-a & F_B^{\otimes s} \otimes((F_B^{\otimes r}\otimes X)\otimes Y) \ar[d]^-a\\
F_B^{\otimes s} \otimes ((X\otimes F_B^{\otimes s})\otimes Y) \ar[dd]_-s & F_B^{\otimes s} \otimes (F_B^{\otimes r} \otimes(X\otimes Y)) \ar[d]^-a\\
& (F_B^{\otimes s} \otimes F_B^{\otimes r}) \otimes (X\otimes Y) \ar[d]^-s \\
F_B^{\otimes r}\otimes ((F_B^{\otimes s}\otimes X)\otimes Y) \ar[r]_-a & F_B^{\otimes r+s}\otimes(X\otimes Y)
}
\end{displaymath}
By the coherence theorem for symmetric monoidal categories this boils down to checking the two permutations agree, which they do.
\end{proof}

\begin{Thm}
\label{thm:birationality}%
Let $\pi\colon \cat{K}\to \SKB$ be the quotient with respect to~$B$ and consider $\xi_B\colon F_B\to \unit$ the  fibre of the coevaluation map $\eta\colon\unit\to B^\vee\otimes B$ as in \Cref{cons:F_B}. Let $C=\cone(\xi_B\potimes{2})$ in~$\cat{K}$. Let $U=U_{\cat{K}}(B)$ the open complement of~$\supp_{\cat{K}}(B)$ and $V=U_{\SKB}(\pi(C))$ the open complement of $\supp_{\SKB}(\pi(C))$. Then $\pi$ restricts to a $\otimes$-triangulated equivalence
\[
\pi\restr{U}\colon\cat{K}(U)\isoto \SKB(V)\,.
\]
In particular, $\Spc(\pi\restr{U})$ is a homeomorphism between~$V$ and~$U$.
\end{Thm}

We shall need a general tt-fact:
\begin{Lem}[{\cite[\S\,2]{Balmer10b}}]
\label{lem:nilp}%
Let $f\colon F\to \unit$ be a morphism in a rigid tt-category such that $f\potimes{n}\otimes\cone(f)=0$ for some~$n\ge 1$. Then the tt-ideal where~$f$ is nilpotent $\cat{J}:=\SET{X\in\cat{K}}{f\potimes{n}\otimes X=0\textrm{ for }n\gg0}$ coincides with $\thick^\otimes(\cone(f))$.
\end{Lem}

\begin{proof}
By \cite[Prop.2.12]{Balmer10b}, this $\cat{J}$ is a tt-ideal, hence our assumption gives $\thick^\otimes(\cone(f))\subseteq\cat{J}$. The reverse inclusion follows from \cite[Prop.\,2.14]{Balmer10b}.
\end{proof}

\begin{proof}[Proof of \Cref{thm:birationality}]
We want to apply \Cref{cor:generalBS} to $\cat{M}=\thick(\proj(\scE_B))$ as announced in \Cref{rem:K(U)}, where we already saw that $\cat{M}=\thick^\otimes(B)=\cat{K}_{\supp(B)}$. So, we need to identify $\pi(\cat{M})=\SET{Y\in\SKB}{Y\simeq \pi(X) \textrm{ for }X\in\cat{M}}$. The cone on $f=\xi_B$ is~$B\otimes B^\vee$ by~\eqref{eq:F1B}, and $f\otimes\cone(f)=0$ by~\eqref{eq:B@xi_B=0}, so we may apply \Cref{lem:nilp}. Noting that $B$ and $B\otimes B^\vee$ generate the same tt-ideal of~$\cK$, namely~$\cat{M}$, we learn that~$\cat{M}$ is exactly the locus of nilpotence of~$\xi_B$:
\[
\cat{M}=\SET{X\in\cat{K}}{\xi_B\potimes{n}\otimes X=0\textrm{ for }n\gg0}.
\]
We claim that its image in~$\SKB$ is simply the locus of nilpotence of the image of~$\xi_B$:
\begin{equation}
\label{eq:aux-pi(M)}%
\pi(\cat{M})=\SET{Y\in\SKB}{\pi(\xi_B)\potimes{n}\otimes Y\textrm{ for }n\gg0}.
\end{equation}
One inclusion is trivial: If $\xi_B\potimes{n}\otimes X=0$ then applying the tensor-functor~$\pi$ we have $\pi(\xi_B)\potimes{n}\otimes \pi(X)=0$. Conversely, suppose that $\pi(\xi_B)\potimes{n}\otimes \pi(X)=0$ for some $n\gg0$ and let us show that $\xi_B\potimes{n\mathbf{+1}}\otimes X=0$. The assumption is that the morphism $\xi_B\potimes{n}\otimes X$ vanishes under~$\pi$, hence factors via some object $B\otimes Z$, say $\xi_B\potimes{n}\otimes X=\beta\alpha\colon F_B\potimes{n}\otimes X\oto{\alpha} B\otimes Z\oto{\beta} X$. In that case, the following commutative diagram
\[
\xymatrix{
F_B\potimes{n+1}\otimes X\ar@/^2em/[rrrr]^-{\xi_B\potimes{n+1}\otimes X} \ar[rr]_-{F_B\otimes\xi_B\potimes{n}\otimes X} \ar[rd]_-{F_B\otimes\alpha}
&& F_B\otimes X \ar[rr]_-{\xi_B\otimes X}
&& X
\\
& F_B\otimes B\otimes Z \ar[ru]_(.6){F_B\otimes\beta} \ar[rr]^-{\xi_B\otimes B\otimes Z}_-{=0\textrm{ by (\ref{eq:B@xi_B=0})}}
&& B\otimes Z \ar[ru]_-{\beta}
}
\]
proves the claim. Hence we have~\eqref{eq:aux-pi(M)}. We now want to use \Cref{lem:nilp} `backwards', for $\pi(\xi_B)$ in the rigid category~$\SKB$. For that, we need to know  the cone on $\pi(\xi_B)$, which is $\pi(\cone(\xi_B\potimes{2}))=\pi(C)$ by \Cref{thm:triangulation-SKB}\,\eqref{it:ucone}. And we need to know that $\pi(\xi_B)$ is nilpotent on its cone. In fact, one can prove directly that already in~$\cat{K}$ we have $\xi_B\potimes{2}\otimes C=0$. In any case, since $C=\cone(\xi\potimes{2})\in\thick^\otimes(\cone(\xi_B))$ by the octahedron on~$\xi_B\potimes{2}=\xi_B\circ (F_B\otimes \xi_B)$, we have $C\in\cat{M}$ and therefore $\xi_B$ is nilpotent on~$C$. The conclusion of \Cref{lem:nilp} for $\pi(\xi_B)$ is that
\[
\pi(\cat{M})=\thick^{\otimes}_{\SKB}(C).
\]
\Cref{cor:generalBS} then tells us that $\pi$ induces an equivalence of Verdier quotients
\[
\cat{K}\Verd\thick^{\otimes}(B)=\cat{K}\Verd\cat{M}\isoto \SKB\Verd\pi\cat{M}=\SKB\Verd\thick^\otimes(C).
\]
Idempotent-completing both sides gives the equivalence $\cat{K}(U)\cong \SKB(V)$ by definition of the open complements~$U$ and~$V$. Finally we have $U\cong\Spc(\cat{K}(U))\cong\Spc(\SKB(V))\cong V$.
\end{proof}

\threestars

There is a more general question of understanding the space $\Spc(\SKB)$ outside of the open $V$ determined in \Cref{thm:birationality}. Note that we have a tt-version of \Cref{prop:Jon}, which indicates that $\Spc(\SKB)$ might be rather big, as there tends to be many more additive tensor-ideals than triangulated tensor-ideals.

\begin{Cor}\label{cor:Jon}
Let $G\colon \cat{K}\to \cat{N}$ be a tt-functor such that every triangle in $\scE_B$ is $G$-split. Let $\cat A\subseteq \cat{N}$ be a summand-closed additive subcategory which contains $GB$ and is closed under tensoring in~$\cat{N}$. Then $\pi (G^{-1}\cat A)$ is a thick $\otimes$-ideal of $\SKB$.
\qed
\end{Cor}

\begin{Rem}
\label{rem:pi^!}%
Given an invertible object $u\in \cat{K}$ in a tt-category we can define a graded commutative ring
\begin{displaymath}
\RKb = \RKub = \Hom_{\cat{K}}(\unit, u^{\otimes\sbull}),
\end{displaymath}
and a spectral (\ie quasi-compact) continuous map
\begin{displaymath}
\rho^\sbull_u \colon \Spc \cat{K} \to \Spec^\sbull(\RKub)
\end{displaymath}
where $\Spec^\sbull$ denotes the homogeneous spectrum. This map is defined by sending a tt-prime $\cat P\in \Spc \cat{K}$ to the homogeneous prime ideal generated by the elements
\begin{displaymath}
\{f\colon \unit \to u^{\otimes n}\;\vert\; \cone(f)\notin \cat P\}.
\end{displaymath}
as $n$ varies.

The prototypical non-trivial invertible object that exists in any tt-category is $\Sigma\unit$ and so we get a comparison map
\begin{displaymath}
\rho_{\cat{K}}^\sbull \colon \Spc \cat{K} \to \Spec^\sbull(\RKb)
\end{displaymath}
where $\RKb=\Rcatb{K,\Sigma\unit}$ is the graded endomorphism ring of the unit.

Of course, this construction is natural in the sense that every tt-functor $F\colon \cat{K}\to \cat L$ induces a ring homomorphism $F\colon \RKub \to \RLvb$, where $v=F(u)$, and thus a map on homogeneous spectra. The functor $F$ also induces a spectral map $F^\ast\colon\Spc\cat L\to \Spc \cat{K}$ and all of this fits into a commutative square
\begin{displaymath}
\xymatrix{
\Spc \cat L \ar[rr]^-{F^\ast} \ar[d]_-{\rho_{\cat L, v}^\sbull} && \Spc \cat{K} \ar[d]^-{\rho_{\cat{K},u}^\sbull} \\
\Spec \RLvb \ar[rr]_-{\Spec^\sbull F} && \Spec \RKub
}
\end{displaymath}

Now let $B$ be any rigid object and $\pi\colon \cat{K} \to \SKB =\cat{K}/\add(B\otimes\cat{K})$ be the corresponding relative category. Set $v= \underline{\Sigma}\unit$. Note that this invertible is not $\Sigma_B(\unit)$. However, the functor $\pi$ yields a graded homomorphism $\pi\colon \RKb\to \Rcat{\SKB,v}^\sbull$ and therefore a spectral continuous map $\Spec^\sbull(\pi)\colon \Spec^\sbull(\Rcat{\SKB,v}^\sbull)\to \Spec^\sbull(\RKb)$.

If we suppose that $\rho_{\cat{K}}^\sbull\colon \Spc\cat{K}\to \Spec^\sbull \RKb$ is a homeomorphism, then there exists a unique spectral map $\pi^!\colon \Spc \SKB \to \Spc \cat{K}$ making the following square commute
\begin{equation}
\label{eq:comparison-map}%
\vcenter{
\xymatrix{
\Spc \SKB \ar[rr]^-{\pi^!} \ar[d]_-{\rho_{\SKB, v}^\sbull} && \Spc \cat{K} \ar[d]^-{\rho_{\cat{K}}^\sbull}_-{\cong} \\
\Spec^\sbull\Rcat{\SKB,v}^\sbull \ar[rr]_-{\Spec^\sbull(\pi)} && \Spec^\sbull\RKb
}}
\end{equation}
simply defined to be the composite
\begin{displaymath}
\pi^!:={\rho^\sbull_{\cat{K}}}^{-1} \circ \Spec^\sbull(\pi) \circ \rho_{\SKB, v}^\sbull
\end{displaymath}
Each map is spectral, so the composite is as well, and unicity is clear.

The assumption that the comparison map $\rho_{\cat{K}}^\sbull\colon \Spc\cat{K}\to \Spec^\sbull \RKb$ is a homeomorphism is intrinsic to~$\cat{K}$ and has really nothing to do with the relative stable category construction~$\SKB$. The above diagram therefore suggests that one could hope to define $\pi^!\colon \Spc(\SKB)\to \SpcK$ in general. In other words, although $\pi\colon \cat{K}\to \SKB$ is \emph{not} a tt-functor, its seems to induce some map on spectra, that extends the birational homeomorphism of \Cref{thm:birationality}.
\end{Rem}



\begin{thebibliography}{CPW98}

\bibitem[Add16]{Addington16}
Nicolas Addington.
\newblock New derived symmetries of some hyperk\"{a}hler varieties.
\newblock {\em Algebr. Geom.}, 3(2):223--260, 2016.

\bibitem[Bal05]{Balmer05a}
Paul Balmer.
\newblock The spectrum of prime ideals in tensor triangulated categories.
\newblock {\em J. Reine Angew. Math.}, 588:149--168, 2005.

\bibitem[Bal10a]{Balmer10b}
Paul Balmer.
\newblock Spectra, spectra, spectra -- tensor triangular spectra versus
  {Z}ariski spectra of endomorphism rings.
\newblock {\em Algebr. Geom. Topol.}, 10(3):1521--1563, 2010.

\bibitem[Bal10b]{BalmerICM}
Paul Balmer.
\newblock Tensor triangular geometry.
\newblock In {\em International {C}ongress of {M}athematicians, {H}yderabad
  (2010), {V}ol. {II}}, pages 85--112. Hindustan Book Agency, 2010.

\bibitem[Bal11]{Balmer11}
Paul Balmer.
\newblock Separability and triangulated categories.
\newblock {\em Adv. Math.}, 226(5):4352--4372, 2011.

\bibitem[BCR97]{BensonCarlsonRickard97}
David~J. Benson, Jon~F. Carlson, and Jeremy Rickard.
\newblock Thick subcategories of the stable module category.
\newblock {\em Fund. Math.}, 153(1):59--80, 1997.

\bibitem[BCS19]{BalandChirvasituStevenson19}
Shawn Baland, Alexandru Chirvasitu, and Greg Stevenson.
\newblock The prime spectra of relative stable module categories.
\newblock {\em Trans. Amer. Math. Soc.}, 371(1):489--503, 2019.

\bibitem[Bel00]{Beligiannis00}
Apostolos Beligiannis.
\newblock Relative homological algebra and purity in triangulated categories.
\newblock {\em J. Algebra}, 227(1):268--361, 2000.

\bibitem[Bel13]{Beligiannis13}
Apostolos Beligiannis.
\newblock Rigid objects, triangulated subfactors and abelian localizations.
\newblock {\em Mathematische Zeitschrift}, 274(3-4):841--883, 2013.

\bibitem[BM94]{BeligiannisMarmaridis94}
Apostolos Beligiannis and Nikolaos Marmaridis.
\newblock Left triangulated categories arising from contravariantly finite
  subcategories.
\newblock {\em Comm. Algebra}, 22(12):5021--5036, 1994.

\bibitem[BS19]{BalandStevenson19}
Shawn Baland and Greg Stevenson.
\newblock Comparisons between singularity categories and relative stable
  categories of finite groups.
\newblock {\em J. Pure Appl. Algebra}, 223(3):948--964, 2019.

\bibitem[B{\"u}h10]{Buehler10}
Theo B{\"u}hler.
\newblock Exact categories.
\newblock {\em Expo. Math.}, 28(1):1--69, 2010.

\bibitem[Car18]{Carlson18}
Jon~F. Carlson.
\newblock Thick subcategories of the relative stable category.
\newblock In {\em Geometric and topological aspects of the representation
  theory of finite groups}, volume 242 of {\em Springer Proc. Math. Stat.},
  pages 25--49. Springer, Cham, 2018.

\bibitem[Chr98]{Christensen98}
J.~Daniel Christensen.
\newblock Ideals in triangulated categories: phantoms, ghosts and skeleta.
\newblock {\em Adv. Math.}, 136(2):284--339, 1998.

\bibitem[CPW98]{CarlsonPengWheeler98}
Jon~F. Carlson, Chuang Peng, and Wayne~W. Wheeler.
\newblock Transfer maps and virtual projectivity.
\newblock {\em J. Algebra}, 204(1):286--311, 1998.

\bibitem[DS{\v S}17]{DellAmbrogioStevensonStovicek17}
Ivo Dell'Ambrogio, Greg Stevenson, and Jan {\v S}{\v t}ov{\'\i}{\v c}ek.
\newblock Gorenstein homological algebra and universal coefficient theorems.
\newblock {\em Math. Z.}, 287(3-4):1109--1155, 2017.

\bibitem[Hap88]{Happel88}
Dieter Happel.
\newblock {\em Triangulated categories in the representation theory of
  finite-dimensional algebras}, volume 119 of {\em LMS Lecture Note}.
\newblock Cambr.\ Univ.\ Press, Cambridge, 1988.

\bibitem[HPS97]{HoveyPalmieriStrickland97}
Mark Hovey, John~H. Palmieri, and Neil~P. Strickland.
\newblock Axiomatic stable homotopy theory.
\newblock {\em Mem. Amer. Math. Soc.}, 128(610), 1997.

\bibitem[J{\o}r10]{Jorgensen10}
Peter J{\o}rgensen.
\newblock Quotients of cluster categories.
\newblock {\em Proc. Roy. Soc. Edinburgh Sect. A}, 140(1):65--81, 2010.

\bibitem[Kel07]{Keller2007differential}
Bernhard Keller.
\newblock On differential graded categories.
\newblock In {\em Proceedings of the {I}nternational {C}ongress of
  {M}athematicians {M}adrid, {A}ugust 22--30, 2006, {V}ol. {II}}, pages
  151--190. {E}uropean {M}athematical {S}ociety, 2007.

\bibitem[Las11]{Lassueur11}
Caroline Lassueur.
\newblock Relative projectivity and relative endotrivial modules.
\newblock {\em J. Algebra}, 337:285--317, 2011.

\bibitem[Mey08]{Meyer08b}
Ralf Meyer.
\newblock Homological algebra in bivariant {$K$}-theory and other triangulated
  categories. {II}.
\newblock {\em Tbil. Math. J.}, 1:165--210, 2008.

\bibitem[MN10]{MeyerNest10}
Ralf Meyer and Ryszard Nest.
\newblock Homological algebra in bivariant {$K$}-theory and other triangulated
  categories. {I}.
\newblock In {\em Triangulated categories}, volume 375 of {\em London Math.
  Soc. Lecture Note Ser.}, pages 236--289. Cambridge Univ. Press, Cambridge,
  2010.

\bibitem[Mor65]{Morita65}
Kiiti Morita.
\newblock Adjoint pairs of functors and {F}robenius extensions.
\newblock {\em Sci. Rep. Tokyo Kyoiku Daigaku Sect. A}, 9:40--71 (1965), 1965.

\bibitem[Nee20]{neeman2020metrics}
Amnon Neeman.
\newblock Metrics on triangulated categories.
\newblock {\em Journal of Pure and Applied Algebra}, 224(4):106--206, 2020.

\bibitem[WZ18]{WangZhang18}
Lizhong Wang and Jiping Zhang.
\newblock Relatively stable equivalences of {M}orita type for blocks.
\newblock {\em J. Pure Appl. Algebra}, 222(9):2703--2717, 2018.

\end{thebibliography}

\end{document}